\documentclass{article}
\usepackage[british]{babel}
\usepackage{color, amsmath, amssymb, amsthm, stmaryrd, graphicx, wrapfig, url, calc}
\usepackage[all]{xy}
\usepackage[a4paper,body={16cm,23cm}]{geometry}

\definecolor{dblue}{rgb}{0,0,0.7}
\newtheoremstyle{mythm}{11pt}{11pt}{\it\color{dblue}}{}{\bf\color{dblue}}{.}{ }{}
\theoremstyle{mythm}
\newtheorem{thm}{Theorem}
\newtheorem{prop}[thm]{Proposition}
\newtheorem{cor}[thm]{Corollary}
\newtheorem{lem}[thm]{Lemma}

\theoremstyle{definition}
\newtheorem*{defi}{Definition}

\renewcommand{\geq}{\geqslant}
\renewcommand{\leq}{\leqslant}

\DeclareMathOperator{\im}{im}
\DeclareMathOperator{\tr}{tr}

\DeclareMathOperator{\coker}{coker}

\DeclareMathOperator{\End}{End}
\DeclareMathOperator{\Aut}{Aut}
\DeclareMathOperator{\Gal}{Gal}
\DeclareMathOperator{\GL}{GL}
\DeclareMathOperator{\SL}{SL}
\DeclareMathOperator{\PGL}{PGL}
\DeclareMathOperator{\divi}{Div}
\DeclareMathOperator{\Pic}{Pic}
\DeclareMathOperator{\Jac}{Jac}

\newcommand{\FF}{\mathbb{F}}
\newcommand{\QQ}{\mathbb{Q}}
\newcommand{\ZZ}{\mathbb{Z}}
\newcommand{\PP}{\mathbb{P}}
\newcommand{\CC}{\mathbb{C}}

\newcommand{\Xn}{X_{\textnormal{nsp}}}
\newcommand{\Xs}{X_{\textnormal{sp}}}
\newcommand{\Yn}{Y_{\textnormal{nsp}}}

\newcommand{\pin}{\pi_{\textnormal{nsp}}}
\newcommand{\pis}{\pi_{\textnormal{sp}}}
\newcommand{\jn}{j_{\textnormal{nsp}}}
\newcommand{\js}{j_{\textnormal{sp}}}
\newcommand{\ion}{\iota_{\textnormal{nsp}}}
\newcommand{\ios}{\iota_{\textnormal{sp}}}
\newcommand{\Xon}{X_{0,\textnormal{nsp}}}
\newcommand{\piin}{\tilde{\pi}_{\textnormal{nsp}}^{+}}

\newcommand{\Fp}{\FF_{\! p}}
\newcommand{\Fps}{\FF_{\!p^2}}
\newcommand\ma[4]{\begin{smallmatrix} #1 & #2 \\ #3 & #4 \end{smallmatrix}}
\newcommand{\antip}{\leftrightarrowtriangle}
\newcommand{\leg}[2]{\genfrac(){0.2pt}0{#1}{#2}}

\begin{document}

\title{A moduli interpretation for the non-split Cartan modular curve}

\author{Marusia Rebolledo and Christian Wuthrich}

\maketitle

\begin{abstract}
 Modular curves like $X_0(N)$ and $X_1(N)$ appear very frequently in arithmetic geometry. While their complex points are obtained as a quotient of the upper half plane by some subgroups of $\SL_2(\ZZ)$, they allow for a more arithmetic description as a solution to a moduli problem. We wish to give such a moduli description for two other modular curves, denoted here by $\Xn(p)$ and $\Xn^{+}(p)$ associated to non-split Cartan subgroups and their normaliser in $\GL_2(\Fp)$. These modular curves appear for instance in Serre's problem of classifying all possible Galois structures of $p$-torsion points on elliptic curves over number fields. We give then a moduli-theoretic interpretation and a new proof of a result of Chen~\cite{chen1,chen2}.
\end{abstract}

\section{Introduction}

 Let $p$ be an odd prime. Let $Y(p)$ be the affine modular curve classifying elliptic curves with full level $p$ structure. The completed modular curve $X(p)$ classifies generalised elliptic curves  with full level $p$ structure.  Those two curves admit integral models over the ring of integers of the cyclotomic field $\QQ(\zeta_p)$. See ~\cite{deligne_rapoport} and~\cite{katz_mazur}. The modular curve $X(p)$ comes equipped with a natural action by $\GL_2(\Fp)$. For any subgroup $\mathcal{H}$ of $\GL_2(\Fp)$ the quotient $X(p)/\mathcal H$ defines an algebraic curve $X_{\mathcal H}$ over $\QQ(\zeta_p)^{\det(\mathcal H)}$.  Hence the points on $X_{\mathcal H}$ over an algebraically closed field $\bar k$ of characteristic different from $p$ are $\mathcal H$-orbits of $\bar k$-points of $X(p)$. However in some interesting cases, there is a nice description of a moduli problem for $X_{\mathcal H}$ too.

 As an example, we explain the case when $\mathcal{H}$ is the Borel subgroup $\mathcal{B} = \bigl\{ (\ma{*}{*}{0}{*})\bigr\}$ in $\GL_2(\Fp)$. First, the points in $Y(p)(\bar k)$ are $\bar k$-isomorphism classes of pairs $\bigl(E,(P,Q)\big)$ where $E/\bar k$ is an elliptic curve and $(P,Q)$ form a basis of $E[p]$. For a fixed $E$, all the pairs $(P',Q')$ in the $\mathcal{B}$-orbit of $(P,Q)$ are such that $P'$ is in the subgroup $C$ generated by $P$. Hence the $\bar k$-points on the quotient curve $Y_{\mathcal B}$ can be identified with $\bar k$-isomorphism classes of pairs $(E,C)$ with $E$ again an elliptic curve defined over $\bar k$ and $C$ a cyclic subgroup of order $p$ in $E[p]$. The latter description is now independent of the initial choice of the Borel subgroup $\mathcal B$ in $\GL_2(\Fp)$ and only uses the geometry of $E$. The curve $X_{\mathcal B}$ is usually denoted by $X_0(p)$.

 Another example is the quotient by the split Cartan subgroup which consists of diagonal matrices in $\GL_2(\Fp)$. The corresponding curve, denoted here by $\Xs(p)$,  parametrises $\bar k$-isomorphism classes $\bigl(E,(A,B)\bigr)$ of generalised elliptic curves $E$ endowed with two distinct cyclic subgroups $A$ and $B$ of order $p$ in $E$. For its normaliser $\mathcal S$, the corresponding curve $X_{\mathcal S} = \Xs^{+}(p)$ classifies generalised elliptic curves with an unordered pair $\{A,B\}$ of cyclic subgroups $A$ and $B$ of order $p$. All these cases are easy to describe because the subgroups $\mathcal H$ can be defined as the stabiliser of some object under a natural action of $\GL\bigl(E[p]\bigr)$.

 In view of Serre's problem to classify the possible Galois module structure of the $p$-torsion of an elliptic curve  over a number field, there are two further modular curves of importance. The aim of this paper is to give a good moduli description for those, namely when $\mathcal{H}$ is a non-split Cartan subgroup or a normaliser of a non-split Cartan subgroup in $\GL_2(\Fp)$. We will denote the corresponding modular curves by $\Xn(p)$ and  $\Xn^{+}(p)$ respectively. See the start of Section~\ref{description_sec} for detailed definitions. These curves have been studied for instance by Ligozat~\cite{ligozat}, Halberstadt~\cite{halberstadt}, Chen~\cite{chen1,chen2}, Merel and Darmon~\cite{merel,darmon_merel} and Baran~\cite{baran}.

 In our description, the modular curve $\Xn^{+}(p)$ will classify elliptic curves endowed with a level structure that we call a \emph{necklace}. Roughly speaking, a necklace is a regular $(p+1)$-gon whose corners, called \emph{pearls}, are all cyclic subgroups of order $p$ in $E$ and such that there is an element in $\PGL\bigl(E[p]\bigr)$ that turns this necklace by one pearl. This will not depend on the choice of a non-split Cartan subgroup.

 In Section~\ref{description_sec}, we define these necklaces and give our moduli description in detail (as well as an alternative and more geometric description using the cross-ratio in $\PP\bigl(E[p]\bigl)$). We also compare this to other moduli descriptions in the literature. In this article, we exclude the case where the base field is of characteristic $p$. In other words, we are treating the modular curves only over $\ZZ[\tfrac{1}{p}]$. We expect however that the moduli problem given here will help to understand the special fibre at $p$, too.

 Section~\ref{geometry_sec} shows how classical results about the geometry of $\Xn(p)$ can be proven using this moduli interpretation. For instance, we can count the number of elliptic points, describe the cusps and the degeneracy maps.

 In~\cite{chen1,chen2}, Chen showed that there is an isogeny between the Jacobian of $\Xs^{+}(p)$ and the product of the Jacobians of  $X_0(p)$ and $\Xn^{+}(p)$. In Section~\ref{chen_sec}, we give a new proof of this theorem using necklaces. This gives an explicit and geometric vision of the maps involved. We conclude the paper with some numerical data related to Chen's Theorem.

 Since the prime $p$ is fixed throughout the paper, we will now omit it from the notations and only write $\Xn$ and $\Xn^{+}$. It is to note that there should be no real difficulty in generalising our moduli description to composite levels $N$. With view on the problem of Serre to classify the Galois structure of $p$-torsion subgroups of elliptic curves over $\QQ$,  prime levels are maybe the most interesting.


\subsection*{Notations}

 The following is a list of modular curves that appear in this paper and the notations we frequently use. The definitions will be given later. See also Sections~\ref{degeneracy_sec} and~\ref{chen_def_subsec} for degeneracy maps and correspondences between them.

 \begin{center}
   \begin{tabular}{cll}
      Symbol    & Description of $\mathcal H < \GL_2(\Fp)$   & Level structure  \\
          \hline && \\[-2ex]
      $X(p)$    & Full level structure, $\mathcal H = \{1\}$ & Basis $(P,Q)$ of $E[p]$  \\
       $ X_{\mathcal A}$ & Scalar matrices $\mathcal A$ & Distinct triple $(A,B,C)$ in $\PP\bigl(E[p]\bigr)$  \\
      $X_0$     & A Borel subgroup $\mathcal B$ & Subgroup $C\in\PP\bigl(E[p]\bigr)$  \\
      $\Xs$     & A split Cartan  & Distinct pair $(A,B)$ in $\PP\bigl(E[p]\bigr)$ \\
      $\Xs^{+}$ & Normaliser of a split Cartan  $\mathcal S$ & Non-ordered pair $\{A,B\}\subset \PP\bigl(E[p]\bigr)$ \\
       $\Xn$     & A non-split Cartan & Oriented necklace $\mathfrak v$ \\
      $\Xn^{+}$ & Normaliser of a non-split Cartan $\mathcal N$ & Necklace $\mathfrak v$
   \end{tabular}
 \end{center}

 Matrices in $\GL_2(\Fp)$ will be written as $\bigl(\ma{\cdot}{\cdot}{\cdot}{\cdot}\bigr)$, while their class in $\PGL_2(\Fp)$ will be represented by$\bigl[\ma{\cdot}{\cdot}{\cdot}{\cdot}\bigr]$. As in the introduction, we will denote the action of $\GL_2(\Fp)$ over $X(p)$ on the right. Instead, if $V$ is a two-dimensional  $\Fp$-vector space,  the action on $V$ of $\GL(V)$  or, after a choice of basis, of $\GL_2(\Fp)$,  will be considered on the left. 


\section{The moduli problem of necklaces}\label{description_sec}

\subsection{Non-split Cartan subgroups and their modular curves}\label{cartan_subsec}

 We refer to~\cite{serre} for definitions and results about non-split Cartan subgroups and Dixon's classification of maximal subgroups of $\GL_2(\Fp)$ and just briefly recall some facts. The group $\GL_2(\Fp)$ acts on the right on $\PP^1(\Fps)$ by $(x:y)\bigl(\ma{a}{b}{c}{d} \bigr)= (ax+cy:bx+dy)$. Any non-split Cartan subgroup of $\GL_2(\Fp)$ can be defined as the stabiliser  $\mathcal H_{\alpha}$ of $(1:\alpha)$ in $\PP^1(\Fps)\setminus \PP^1(\Fp)$ for a choice of $\alpha \in \Fps\setminus\Fp$. We see that $\mathcal H_{\alpha}$ has order $p^2-1$ as the action of $\GL_2(\Fp)$ is transitive on $\PP^1(\Fps)\setminus \PP^1(\Fp)$.

 Alternatively, we can consider the basis $(1,\alpha)$ of $\Fps$ as a $\Fp$-vector space. Then we claim that $\mathcal H_\alpha$ is equal to the image of the map $i_\alpha \colon \Fps^{\times}\to \GL_2(\Fp)$ sending $\beta$ to the matrix which represents the multiplication by $\beta$ on $\Fps$ written in basis $(1,\alpha)$.
 Indeed, let $\beta = x +y \alpha\in\Fps^{\times}$ with $x,y\in\Fp$. If $X^2 - t X +n$ is the minimal polynomial of $\alpha$ over $\Fp$, then
 \begin{equation*}
   i_\alpha(\beta) = \begin{pmatrix} x & -ny \\ y & x+ty \end{pmatrix}
 \end{equation*}
 and so $(1:\alpha)i_\alpha(\beta) = (x+y\alpha : -ny + (x+ty)\alpha) = (\beta:\beta \alpha) = (1:\alpha) $. So the image of $i_\alpha$ is contained in $\mathcal H_\alpha$ and they are equal because they are of the same size.

 Given a choice of a non-split Cartan subgroup $\mathcal H$, we define the modular curve $\Xn$ as the quotient $X_{\mathcal H}$. Note that the quotient does not depend on the choice of $\mathcal H$ as these subgroups are all conjugate. However the description of points on $\Xn$ as $\mathcal H$-orbits do.

 The normaliser $\mathcal N$ of a non-split Cartan subgroup $\mathcal H$ in $\GL_2(\Fp)$ contains $\mathcal H$ with index $2$. It can be viewed as adding the image under $i_\alpha$ of the conjugation map in $\Gal(\Fps/\Fp)$ on $\Fps^{\times}$. The corresponding quotient $X_{\mathcal N}$ will be denoted by $\Xn^{+}$.


\subsection{Necklaces}\label{subsec:necklaces}

 Let $\gamma$ be a multiplicative generator of $\Fps^{\times}$. For any $2$-dimensional $\Fp$-vector space $V$, we define $\mathcal{C}_{\gamma}$ to be the conjugacy class in $\PGL(V)$ of all elements $h$ which have a representative in $\GL(V)$ whose characteristic polynomial is equal to the minimal polynomial of $\gamma$. In other words, all representatives of $h\in \mathcal{C}_{\gamma}$ have an eigenvalue in $\Fp^{\times}\cdot\gamma$. The set $\mathcal{C}_\gamma$ depends only on $\Fp^\times.\gamma$.  If $\bar\gamma$ is the conjugate of $\gamma$ over $\Fp$, then $\mathcal{C}_{\bar\gamma}=\mathcal{C}_{\gamma}$. If a basis of $V$ is chosen then $\mathcal{C}_{\gamma}$ consist of all classes of matrices $i_{\alpha}(\gamma)$ as $\alpha$ runs through $\Fps\setminus\Fp$. Note that the class of  $i_{\alpha}(\bar\gamma)=i_{\bar{\alpha}}(\gamma) = \operatorname{N}(\gamma)^{-1}\, i_\alpha(\gamma)^{-1}$ is equal to the inverse of class of $i_\alpha(\gamma)$. In particular, in any non-split Cartan subgroup in $\PGL(V)$, there are exactly two generators $h$ and $h^{-1}$ that belong to $\mathcal{C}_{\gamma}$. As $\gamma$ varies, we obtain the $\tfrac{1}{2}\varphi(p+1)$ conjugacy classes of elements of order $p+1$.

 \begin{wrapfigure}{r}{43mm}%
   \vspace{-10pt}\centering%
   \includegraphics[width=4cm]{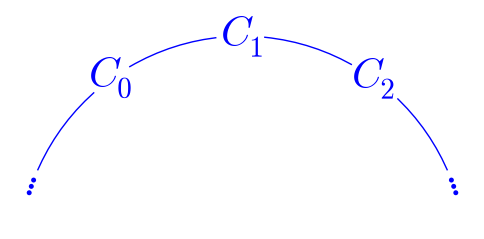}%
   \vspace{-10pt}%
 \end{wrapfigure}
 Let $\bar k$ be an algebraically closed field of characteristic 0 or different from $p$ and let $E/\bar k$ be an elliptic curve. We know that there exists $p+1$ cyclic subgroups of order $p$ in $E[p]$. We will consider  lists  $(C_0,C_1,\dots, C_p)$ of those cyclic subgroups and will say that two such lists are  equivalent if we can obtain one from the other by a cyclic permutation; so $(C_0,C_1,\dots, C_p)$ and $(C_1,C_2,\dots, C_p,C_0)$ are equivalent.

Recall that we will denote on the left the action by homography of  $\PGL(E[p])$ on $\PP(E[p])$.
 \begin{defi}
   An equivalence class $(C_0,C_1,\dots, C_p)$ is called an \textit{oriented $\gamma$-necklace} of $E$ if there exists an element $h\in\mathcal{C}_\gamma\subset \PGL(E[p])$ such that $h(C_i) =  C_{i+1}$ for all $i=0,\dots, p-1$.
 \end{defi}

 If $h$ is such an element, then we must also have $h(C_p) = C_0$ as $h$ is of order $p+1$. Note also that if $(C_0,C_1,\dots, C_p)$ is an oriented necklace with a certain $h\in \mathcal{C}_\gamma$, then so is $(C_p,C_{p-1},\dots, C_0)$ because $h^{-1}\in\mathcal{C}_{\gamma}$.

 Let us consider the dependence on the choice of $\gamma$.

 \begin{lem}\label{gamma_lem}
   Let $\gamma$ and $\gamma'$ be two generators of $\Fps^{\times}$. There is a canonical bijection between oriented $\gamma$-necklaces and oriented $\gamma'$-necklaces.
 \end{lem}

 \begin{proof}
   Since $\Fps^{\times}$ is cyclic, there exists an integer $k\in [0,p^2-1]$ such that $\gamma'=\gamma^k$ and such that $k$ is coprime to $p+1$. In particular $\mathcal{C}_{\gamma'}$ is the set of all $h^k$ with $h\in\mathcal{C}_{\gamma}$. So the requested bijection is given by
   \begin{align*}
     \{\text{oriented $\gamma$-necklaces}\} &\to  \{\text{oriented $\gamma'$-necklaces}\}\\
     (C_0,C_1,\dots, C_p)                   &\mapsto (C_0,C_k,C_{2k},\dots )    \end{align*}
   with the index taken modulo $p+1$.
 \end{proof}

 As a consequence, we may now fix a choice of $\gamma$ for the rest of the paper and call the oriented $\gamma$-necklaces simply \emph{oriented necklaces}.

 In a picture, we arrange the subgroups $C_0$, \dots, $C_{p}$ like pearls on a necklace that can be turned around the neck using the automorphism $h$ of $\PP(E[p])$. If we allow the necklace to be worn in both directions, we get the notion of a necklace without orientation:

 \begin{defi}
   Let $w$ denote the involution defined by
   \begin{equation*}
     w(C_0,C_1,\dots,C_p)=(C_p, C_{p-1},  \dots, C_0)
   \end{equation*}
   which changes the orientation of an oriented necklace. A \textit{necklace} is a $w$-orbit of oriented necklaces $\{\mathfrak v, w(\mathfrak v)\}$.
 \end{defi}

 \begin{lem}\label{g_lem}
   Fix a generator $\gamma$ of $\Fps^{\times}$. Let $C_0$, $C_1$, and $C_2$ be three distinct cyclic subgroups of order $p$ in $E[p]$. Then there exists a unique element $h\in\mathcal{C}_\gamma$ in $\PGL(E[p])$ such that $h(C_0)=C_1$ and $h(C_1)=C_2$.
 \end{lem}

 \begin{proof}
   Choose generators $P_0$ and $P_1$ in $C_0$ and $C_1$ respectively and consider the basis $(P_0,P_1)$ of $E[p]$. Write $t$ and $n$ for the trace and the norm of $\gamma$.  The elements $h$ of $\mathcal C_\gamma$ which verify $h(C_0)=C_1$ are   of the form $[\ma{0}{y}{-ny^{-1}}{t}]$ with $y\in \Fp^{\times}$. Moreover, as $y$ varies, the points $yP_0+tP_1$ form an affine line and hence there is a unique $y\in \Fp^{\times}$ such that $yP_0+tP_1$ belongs to $C_2$. Indeed, note that $t\neq 0$ because if it were, $\gamma$ would be of order dividing $2(p-1)$. Hence, there is a unique $h\in\mathcal C_\gamma$ such that $h(C_0)=C_1$ and $h(C_1)=C_2$.
%
 \end{proof}

 This lemma implies that for any triple $(C_0, C_1, C_2)$ of distinct cyclic subgroups of $E[p]$, there is a unique oriented necklace of the form $(C_0,C_1,C_2,\dots)$. We will denote it by $C_0\to C_1\to C_2$. Similarly there is a unique necklace with consecutive pearls $C_0,C_1,C_2$, which we denote by $C_0 - C_1 - C_2$.

 There is a natural action  of $\PGL(E[p])$ on the set of oriented necklaces by setting $(C_0,\dots C_p).g = \bigl(g(C_0), \dots, g(C_p)\bigr)$ for $g$ in $\PGL(E[p])$. If $h\in\mathcal{C}_{\gamma}$ is such that $h(C_i)=C_{i+1}$ then $ghg^{-1}\in \mathcal{C}_{\gamma}$ can be used to show that $\bigl(g(C_0), \dots, g(C_p)\bigr)$ is indeed an oriented necklace. Since the action of $\PGL(E[p])$ on $\PP(E[p])$ is simply $3$-transitive, Lemma~\ref{g_lem} implies that the action of $\PGL(E[p])$ on oriented necklaces  is transitive. By definition, for every oriented necklace $\mathfrak v$, there exists  $h\in \mathcal{C}_{\gamma}$ fixing it. Therefore, the group generated by $h$, which is a non-split Cartan subgroup in $\PGL(E[p])$ will belong to the stabiliser of $\mathfrak v$. It is clear that this is equal to the stabiliser of $\mathfrak v$. We have shown:

 \begin{cor}\label{gset_cor}
   Let $\mathcal G=\PGL(E[p])$. The set of oriented $\gamma$-necklaces is isomorphic as a $\mathcal G$-set to $\mathcal G/\mathcal H$ where $\mathcal H$ is any non-split Cartan group in $\mathcal G$. Similarly, the set of $\gamma$-necklaces is $\mathcal G$-isomorphic to $\mathcal G/\mathcal N$ for the normaliser of a non-split Cartan group $\mathcal N$ in $\mathcal G$. In particular, there are exactly $p(p-1)$ oriented necklaces and $p(p-1)/2$ necklaces.
 \end{cor}


\subsection{Moduli description}\label{moduli_subsec}

 Let $\mathcal H$ be a non-split Cartan subgroup in $\mathcal G = \GL_2(\Fp)$ and write $\mathcal N$ for its normaliser. Let $\bar k$ be an algebraically closed field of characteristic different from $p$. Recall that  $Y(p)(\bar k)$ classifies the $\bar k$-isomorphism classes of pairs $\bigl(E,(P,Q)\bigr)$  where $E$ is an elliptic curve over $\bar k$ and $(P,Q)$ is an $\Fp$-basis of $E[p]$. The group $\GL_2(\Fp)$ acts on a pair on the right as usual $\bigl(E,(P,Q)\bigr)\cdot \left(\ma{a}{b}{c}{d}\right)  = \bigl(E, (aP + cQ, bP+dQ)\bigr)$. A point in $Y_{\mathcal H}(\bar k)$ is an orbit under this action by the non-split Cartan subgroup $\mathcal H$. For a given elliptic curve $E/\bar k$, the $\mathcal H$-orbits of triples is a $\mathcal G$-set isomorphic to $\mathcal G/\mathcal H$. Corollary~\ref{gset_cor} has shown us that the set of oriented necklaces on $E$ is also isomorphic to $\mathcal G/\mathcal H$. Hence we have:

 \begin{prop}
   Let $\mathcal H$ be a non-split Cartan subgroup in $\GL_2(\Fp)$. There is a bijection between the points in $Y_{\mathcal H}(\bar k)$ and the set of $\bar k$-isomorphism classes of pairs $(E,\mathfrak v)$ composed of an elliptic curve $E/\bar k$ together with an oriented necklace $\mathfrak v$ in $E$. Similarly $Y_{\mathcal N}(\bar k)$ consists of pairs $(E,\mathfrak v)$ where $\mathfrak v$ is a necklace in $E$.
 \end{prop}

In Section~\ref{galois_subsec}, we will give the description of $k$-rational points for fields $k$ which are not algebraically closed. We will from now on informally say that $\Yn^{+}$ and $\Yn$ are coarse moduli spaces for the moduli problem of elliptic curves endowed with a necklace and an oriented necklace respectively.

In order to make this precise, we would have to extend the definition of necklaces to elliptic curves over arbitrary schemes.
But, as we will now explain briefly, the most natural definition gives a functor of moduli problem that fails to capture the correct Galois action, making the situation  not entirely satisfactory in this case. Note that one encounters the same problem for the case of split Cartan subgroup, while the analogous situation for $Y_0$ and $Y_1$ works nicely.

For any $\ZZ\bigl[\tfrac{1}{p}\bigr]$-scheme $S$ and elliptic curve $E/S$, we could define a $\gamma$-necklace on $E/S$ to be a complete list of all $p+1$ cyclic $S$-subgroup schemes $(C_0, \dots, C_p)$ of order $p$ in $E[p]$ up to cyclic permutation verifying the condition with respect to $\gamma$ as in Section~\ref{subsec:necklaces}. Consider the moduli problem given by the functor $\mathcal{F}$ associating to a $\ZZ\bigl[\tfrac{1}{p}\bigr]$-scheme $S$, the set of $S$-isomorphism classes of $(E/S,\mathfrak{v})$ where $E/S$ is an elliptic curve and $\mathfrak{v}$ such a $\gamma$-necklace on $E/S$. Then, essentially by the above Proposition, the scheme $\Yn^+$ over $\ZZ\bigl[\tfrac{1}{p}\bigr]$ is a coarse moduli scheme for this functor~$\mathcal{F}$.

However, contrarily to  what happens for the coarse moduli scheme $Y_0$ and a fortiori for the fine moduli scheme $Y_1$, the morphism of functors $\mathcal{F} \to \Yn^+$ does not necessarily give a surjection $\mathcal{F}(k) \to \Yn^+(k)$ when $k$ is a field that is not algebraically closed: For instance $\mathcal{F}(\mathbb{Q})$ is empty for all odd primes $p$, yet $\Yn^{+}(\mathbb{Q})$ will contain points for many primes $p$. Since we wish to describe all points of $\Yn^{+}$, this functor~$\mathcal{F}$ is not the best choice. An option would be to write down a different functor, but that turns out to be cumbersome. In this article, we prefer to view points in $\Yn^{+}(k)$ for a field $k$ as points in $\Yn^{+}(\bar k)$ which are fixed by the absolute Galois group of $k$. They will be described in Section~\ref{galois_subsec} in term of necklaces.

The above problem is also present for the split Cartan subgroup. The description of the $\bar k$-rational points using non-ordered pairs of cyclic subgroups together with the Galois action has nevertheless been used extensively (see for instance~\cite{momose}). The aim of this article is, in a similar spirit, to obtain basic properties of $\Xn$ and $\Xn^{+}$ via our description.

Now, it is very natural to ask about the fibre at $p$ of a good model of $\Xn$ and $\Xn^+$ over $\ZZ$. The naive extension of the above definition of necklace for elliptic curves over $\ZZ$-schemes cannot work as elliptic curves in characteristic $p$ have either one or two distinct cyclic (in the sense of~\cite{katz_mazur}) subgroup schemes of order $p$. To find an appropriate definition of necklaces that would also work for characteristic $p$ should be the topic of a future investigation.

\subsection{The cross-ratio}
As before, the action by homography of $\PGL\bigl(E[p]\bigr)$ on the projective space $\PP\bigl(E[p]\bigr)$ is denoted on the left. We denote the elements of  $\PP^1(\Fp)$ by $\infty=(1:0)$ and $a=(a:1)$ for $a\in\Fp$.

 Let $A,B,C$ be three distinct points in $\PP\bigl(E[p]\bigr)$ and $D\in \PP\bigl(E[p]\bigr)$. Recall that the cross-ratio of $A,B,C,D$ is defined by $[A,B;C,D]=f(D)$ where $f:\PP\bigl(E[p]\bigr)\longrightarrow \PP^1(\Fp)$ is the unique isomorphism such that $f(A)=\infty,$ $f(B)=0$ and $ f(C)=1.$  After a choice of basis of $E[p]$ identifying $\PP\bigl(E[p]\bigr)$ with~$\PP^1(\Fp)$, we get
 \begin{equation*}
  [A,B;C,D] = \frac{A-C}{B-C}\cdot \frac{B-D}{A-D}.
 \end{equation*}
 The last formula is independent of the choice of basis since  the cross-ratio is $\PGL\bigl(E[p]\bigr)$-invariant. If $\mathfrak v=(C_0,C_1,\dots, C_p)$ is an oriented necklace, then $[C_0,C_1;C_2,C_3] = [C_i,C_{i+1}; C_{i+2},C_{i+3}]$ for all $0\leq i\leq p$ with the index taken modulo $p+1$. Hence we can attach a cross-ratio to each necklace. As described above, the action of $\PGL\bigl(E[p]\bigr)$ on oriented necklaces is transitive and hence this cross-ratio $[C_0,C_1;C_2,C_3]$ is the same for all oriented $\gamma$-necklaces.

 \begin{prop}
   Let $\gamma$ be a generator of $\Fps^{\times}$ of trace $t$ and norm $n$. Set $\xi_{\gamma} = t^2/(t^2-n)$. Then a list $(C_0,C_1,\dots, C_p)$ of all distinct cyclic subgroups of order $p$ in $E$ represents a $\gamma$-necklace if and only if $[C_i,C_{i+1};C_{i+2},C_{i+3}]=\xi_{\gamma}$ for all $0\leq i \leq p$ with the index taken modulo $p+1$.
 \end{prop}

 This provides a new possibility of defining necklaces by-passing completely the use of the automorphism group of $E[p]$, but only relying on the projective geometry of $\PP(E[p])$.

 \begin{proof}
   We only need to compute the cross-ratio for one necklace. We take the basis such that $h=[\ma{0}{-n}{1}{t}]$ is in $\mathcal{C}_\gamma$. The necklace now contains the consecutive pearls $\infty$, $0$, $-n/t$ and $-nt/(-n+t^2)$ from which we obtain the above cross-ratio $\xi_{\gamma}$.
 \end{proof}

 Since to each triple $(C_0,C_1,C_2)$ there is a unique $C_3$ such that $[C_0,C_1;C_2,C_3]=\xi_\gamma$, we have a second proof of Lemma~\ref{g_lem}.

\subsection{Relation to other descriptions}\label{merel_subsec}

 We recall a different description of the $\mathcal H_\alpha$-orbits of points in $X(p)$ where $\alpha$ is a choice in $\Fps\setminus\Fp$. See~\cite{merel}. Let $E/\bar k$ be an elliptic curve.  Choose a basis $P_0$, $P_1$ of $E[p]$ and identify $E[p]$ with $\Fps$ via $P_0 \mapsto 1$ and $P_1\mapsto \alpha$.  Any basis $(P,Q)$ of $E[p]$ is equal to $(P_0,P_1)g$ for some $g\in\GL_2(\Fp)$. Consider the $\GL_2(\Fp)$-equivariant map which sends $(P_0,P_1)$ to $(1:\alpha)\in\PP^1(\Fps)\setminus\PP^1(\Fp)$. Since the action of $\mathcal H_\alpha$ is now just the multiplication on $\Fps$, it induces a well defined $\GL_2(\Fp)$-equivariant  map from the set of $\mathcal H_\alpha$-orbits of basis $(P,Q)$ to $\PP^1(\Fps)\setminus\PP^1(\Fp)$. This is a $\GL_2(\Fp)$-equivariant bijection.

 This leads now to a moduli problem description of $\Xn$. Each point in $\Yn(\bar k)$ with $\bar k $ an algebraically closed field of characteristic different from $p$ is a $\bar k$-isomorphism class of $(E,\mathfrak C)$ where $E/\bar k$ is an elliptic curve and $\mathfrak C$ is an element in $\PP\bigl( E[p]\otimes \Fps \bigr) \setminus \PP\bigl(E[p]\bigr)$. The group $\PGL\bigl(E[p]\bigr)$ acts on the left on $\PP\bigl(E[p]\otimes \Fps\bigr)$ by its action on $E[p]$.

We will now give an explicit $\PGL\bigl(E[p]\bigr)$-equivariant bijection between the set of oriented $\gamma$-necklaces of $E$ and  $\PP\bigl( E[p]\otimes \Fps \bigr) \setminus \PP\bigl(E[p]\bigr)$.  Write $n$ and $t$ for the norm and trace of the fixed element $\gamma$ in $\Fps$. Consider the map
\begin{equation}\label{merelade}
\begin{array}{rcl}
\{\text{$\gamma$-necklaces}\} & \longrightarrow &\PP\bigl( E[p] \otimes \Fps \bigr) \\
(C_0, C_1, C_2, \dots) & \mapsto& \langle P\otimes (-\gamma) + Q\otimes 1\rangle
\end{array}
\end{equation}
 where $(P,Q)$ is a basis of $E[p]$ such that $C_0= \langle P \rangle$, $C_1=\langle Q\rangle$ and $C_2=\langle -n P + t Q\rangle $.   Note that such a basis exists because neither $n$ nor $t$ could be zero when $\gamma$ is a multiplicative generator of $\Fps^\times$. We have to show that this map is well-defined. Let $h$ be a generator in the stabiliser of $\mathfrak v$ which belongs to $\mathcal{C}_{\gamma}$. In the basis $(P,Q)$ this element $h$ is represented by the matrix $\bigl[\ma{0}{-n}{1}{t}\bigr]$. Now
 \begin{equation*}\label{compute}
   h\Bigl( P\otimes (-\gamma) + Q\otimes 1 \Bigr) = Q\otimes (-\gamma) + (-nP+tQ)\otimes 1 = (t-\gamma) \cdot \Bigl( P \otimes (-\gamma) + Q\otimes 1 \Bigr)
 \end{equation*}
 as $(t-\gamma )(-\gamma) = \gamma^2 -t \gamma = -n$. This shows that the line in $\PP\bigl( E[p] \otimes \Fps \bigr) $ does not depend on the choices made in the construction. It also is evident from this that the stabiliser of $\mathfrak v$ is equal to the stabiliser of the image. From the construction we see that the map is $\PGL\bigl(E[p]\bigr)$-equivariant. Since the actions are transitive, it  follows that it is surjective and hence bijective.

 Since we have no geometric object linked to $E$ which can be thought of directly as an element in $E[p]\otimes \Fps$, we believe that the moduli problem of necklaces has its advantages.

\medskip
 While finalising this article, we learnt of yet another moduli interpretation given by Kohen and Pacetti in~\cite{kohen_pacetti}: Fix a choice of a quadratic non-residue $\varepsilon$ modulo $p$. They represent each point in $\Yn^+(\bar k )$ by a $\bar k$-isomorphism class of $(E,\phi)$ where $E/\bar k$ is an elliptic curve and $\phi\in\GL\bigl(E[p]\bigr)$ is an element such that $\phi^2$ is the multiplication by $\varepsilon$. See Proposition~1.1 and Remark~1.3 in~\cite{kohen_pacetti}. The following defines a $\PGL\bigl(E[p]\bigr)$-equivariant bijection between the set of such endomorphisms $\phi$ and the set of necklaces on $E$. The endomorphism $\phi$ defines an element of order two in $\PGL\bigl(E[p]\bigr)$ without fixed point; so it belongs to a unique non-split Cartan subgroup $\mathcal H$ whose normaliser is the stabiliser of a necklace $\mathfrak v$. Conversely, every stabiliser of a necklace contains a unique element in $\PGL\bigl(E[p]\bigr)$ that lifts to an element $\phi \in \GL\bigl(E[p]\bigr)$ with $\phi^2 = \varepsilon$.


\section{Describing the geometry and arithmetic with necklaces}\label{geometry_sec}

\subsection{Degeneracy maps}\label{degeneracy_sec}

 Let $\mathcal A$ be the group of scalars  in $\GL_2(\Fp)$ and consider the associated modular curve $X_{\mathcal A}$. Because the group $\PGL_2(\Fp)$ acts sharply $3$-transitive on $\PP^1(\Fp)$, the curve $X_{\mathcal A}$ represents the moduli problem associating to each elliptic curve $E$ a triple of distinct cyclic subgroups $(C_0,C_1,C_2)$ of order $p$ in $E$, which is also called a projective frame in $\PP(E[p])$.

 The map $\pi_{\mathcal A}\colon X(p)\to X_{\mathcal A}$ can be chosen to be the following. Let $n$ and $t$ be the norm and trace of our fixed generator $\gamma$ in $\FF_{p^2}$. To each basis $(P,Q)$ of the $p$-torsion of an elliptic curve $E$, we associate the triple $\bigl( \langle P \rangle, \langle Q\rangle, \langle -n\,P+t\,Q\rangle \bigr)$. From the fact that $t\neq 0$, it is clear that this gives a map $X(p)\to X_{\mathcal A}$. Next we describe the map $\pin\colon X_{\mathcal A} \to \Xn$. We have a natural choice to send the triple $(C_0,C_1,C_2)$ to the unique oriented necklace $C_0\to C_1\to C_2$ given by Lemma~\ref{g_lem}. Similarly, we will send it to the necklace $C_0 - C_1 - C_2$ to define the map $\pin^+ \colon X_{\mathcal A} \to \Xn^+$.

 The advantage of our choices is that $\pin\circ \pi_{\mathcal A}$ provides  an explicit bijection between the set of orbits of isomorphism classes $(E,(P,Q))$ under a particular non-split Cartan subgroup $\mathcal H_0$ and the set of isomorphism classes $(E,\mathfrak v)$ of elliptic curves endowed with an oriented necklace. Let $\mathcal H_0$ be the non-split Cartan subgroup in $\GL_2(\Fp)$ generated by the matrix $h_0 = \bigl(\ma{0}{-n}{1}{t}\bigr)=i_\gamma(\gamma)$, which is an element in our chosen class $\mathcal{C}_{\gamma}$ for $V=\Fp^2$.
 Let $E$ be an elliptic curve and $(P,Q)$ a basis of $E[p]$. Denote by  $h\in\PGL(E[p])$ the element of order $p+1$ defining the necklace $\mathfrak v= \pin\circ \pi_{\mathcal A}(P,Q)$.
 Then, by construction of $\pi_{\mathcal A}$,
 \begin{equation*}
   \pi_{\mathcal A}\bigl( (P,Q) \cdot h_0 \bigr ) = h \cdot \pi_{\mathcal A}(P,Q).
 \end{equation*}
 This insures that the map which sends an orbit  $(E,(P,Q))\,\mathcal H_0$ to $(E,\mathfrak v)$ where $\mathfrak{v}= \pin\circ \pi_{\mathcal A}(P,Q) $ is well defined and it gives the expected bijection.

 Under the map~\eqref{merelade} in Section~\ref{merel_subsec}, identifying necklaces with elements in $\PP\bigl( E[p]\otimes \Fps \bigr) \setminus \PP\bigl(E[p]\bigr)$, the degeneracy map above can also be described as sending the basis $(P,Q)$ of $E[p]$ to the projective line $\langle P\otimes (-\gamma) + Q\otimes 1\rangle$ in $E[p]\otimes \Fps$. This provides an explicit bijection between the set of orbits of isomorphism classes $(E,(P,Q))$ under $\mathcal H_0$ and the set of isomorphism classes $(E,\mathfrak C) $ of  elliptic curves endowed with an element  $\mathfrak C$  in $\PP\bigl( E[p]\otimes \Fps \bigr) \setminus \PP\bigl(E[p]\bigr)$.

 We will see later in Section~\ref{chen_grth_subsec} another naturally defined degeneracy map $\piin\colon X_{\mathcal A}\to \Xn^{+}$.

\subsection{Cusps}\label{cusps_subsec}

The following proposition, quoted (but not proved) in~\cite{serre} Appendix~A.5, can be proved using necklaces:
 \begin{prop}
  The modular curve  $\Xn$ has $p-1$ cusps, each ramified of degree $p$ over the cusp $\infty$ in $X(1)$.
 \end{prop}
 \begin{proof}
   In  order to determine the structure of the cusps, we use the Tate curve $E_q$ over $\QQ(\!(q)\!)$. Formally, one can deduce the proposition using Theorem~10.9.1 in~\cite{katz_mazur} from the fact that a non-split Cartan subgroup of $\PGL_2(\Fp)$ acts transitively on $\PP^1(\Fp)$ and that it contains no non-trivial element from any Borel subgroup. In particular, the formal completion of $\Xn$ along the cusps is the formal spectrum of $\QQ(\zeta)[\![\alpha]\!]$, where $\alpha^p = q$ and $\zeta$ is a $p$-th root of unity.

   However, we can also view it on the necklaces of $E_q$. The Tate curve has a distinguished cyclic subgroup $\mu_p$ of order $p$. Any oriented necklace $\mathfrak v$ can be turned in such a way that $C_0=\mu_p$. The two following pearls $C_1$ and $C_2$ have each a generator which is a $p$-th root of $q$, say $\alpha \zeta^i$ and $\alpha \zeta^j$, respectively, where $0\leq i \neq j < p$. From the action of the inertia group of the extension $\QQ(\!(q)\!)[\alpha,\zeta]$ over $\QQ(\!(q)\!)$, we see that all the $p$ necklaces with a given $i-j\in \Fp^{\times}$ meet at the same cusp in the special fibre at $(q)$.
 \end{proof}

 The cusps are not defined over $\QQ$ but over the cyclotomic field $\QQ(\mu_p)$ only, forming one orbit under the action of the Galois group, despite the fact that $\Xn$ is defined over $\QQ$. See Appendix~A.5 in~\cite{serre}. As a consequence there are $p-1$ choices of embeddings $\Xn \hookrightarrow \Jac(\Xn)$, all defined over $\QQ(\mu_p)$ only and none of them is a canonical choice.

 With the same proof we show that $\Xn^{+}$ has $(p-1)/2$ cusps defined over the maximal real subfield of $\QQ(\mu_p)$.


\subsection{Galois action}\label{galois_subsec}

 Let $k$ be a field of characteristic different from $p$ and write $G_k$ for its absolute Galois group. For any $\sigma$ in $G_k$ and point $x \in \Yn(\bar k)$, represented by the pair $(E,\mathfrak v)$, we define $\sigma(x)$ in the obvious way as the $\bar k$-isomorphism class of the pair $\bigl(E^{\sigma}, \sigma(\mathfrak v)\bigr)$.  Here $\sigma\bigl((C_0,C_1,\dots)\bigr)$ is the necklace $\bigl(\sigma(C_0), \sigma(C_1), \dots \bigr)$. Write $\Yn(k)$ for the elements in $\Yn(\bar k)$ fixed by $G_k$.

 \begin{prop}
   Let $x \in \Yn(k)$. Then there exists a pair $(E,\mathfrak v)$ representing $x$ such that $E$ is defined over $k$. If $j(E) \not\in \{0,1728\}$ then the oriented necklace $\mathfrak v$ is also defined over $k$, in the sense that $\sigma(\mathfrak v) = \mathfrak v$ for all $\sigma\in G_k$. In particular, the image of the residual Galois representation $\bar\rho_p (E)\colon G_k \to \GL\bigl(E[p]\bigr)$ has its image in a non-split Cartan subgroup.
 \end{prop}

 \begin{proof}
   Let $(E, \mathfrak{v})$ be a representation of $x$. So $E^{\sigma}$ is $\bar k$-isomorphic to $E$. As usual $\sigma\bigl(j(E)\bigr) = j(E^{\sigma}) = j(E)$, shows that $j(E)\in k$ and hence we may assume that $E$ is defined over $k$.

   For every $\sigma \in G_k$ there is an automorphism $\Psi_\sigma\in\Aut_{\bar k}(E)$ such that $\Psi_{\sigma}(\mathfrak v) = \sigma(\mathfrak v)$. If $j(E)\not\in\{ 0, 1728\}$ then there are no additional automorphisms besides $[\pm1 ]$. Therefore they all act by scalars on $E[p]$ and thus they act trivially on $\PP\bigl(E[p]\bigr)$. It follows that $\mathfrak v = \sigma(\mathfrak v)$. So the image of $\bar \rho_p(E)$ lands in the stabiliser of $\mathfrak v$, which is a non-split Cartan subgroup.
 \end{proof}

 If $\mathfrak v$ is defined over $k$ then there exists a cyclic extension $L/k$ of degree dividing $p+1$ such that all cyclic subgroups $C$ of $E[p]$ are defined over $L$.

 The analogous statement holds for $\Yn^{+}(k)$: Every point in $\Yn^{+}(k)$ can be represented by a pair $(E,\mathfrak v)$ with $E$ being defined over $k$. If $j(E)\not\in\{ 0, 1728\}$, then the necklace $\mathfrak v$ has also to be defined over $k$ and the residual Galois representation takes values in the normaliser of a non-split Cartan subgroup of $\GL\bigl(E[p]\bigr)$.

 The fibres in $\Yn$ and $\Yn^{+}$ above the points $j=0$ or $j=1728$, contain ramified points and elliptic points. We will discuss the elliptic points   in detail in Section~\ref{elliptic_points_subsec}. A ramified point $x$ in one of those fibres can still be represented by $(E,\mathfrak v)$ with $E$ defined over $\QQ$. However we will show now that the field of definition of $x$ is not equal to the field of definition of $\mathfrak v$ if $p>3$: In both cases, the curve $E$ has complex multiplication and by a general result (see Corollary~5.20 in Rubin's part in~\cite{cetraro}), the image of the Galois representation $\rho\colon G_{\QQ} \to \GL\bigl(E[p]\bigr)$ contains the image of all the automorphisms of $E$. Since $x$ is ramified, there is an automorphism $g$ such that $g(\mathfrak v) \neq \mathfrak v$. Hence there is an element $\sigma\in G_{\QQ}$ that sends $\mathfrak v$ to $g(\mathfrak v)$. Now $\sigma$ does not fix $\mathfrak v$, but it fixes $x$, which is also represented by $\bigl(E, g(\mathfrak v)\bigr)$.


\subsection{A lemma on antipodal pearls and cross-ratios}\label{antipodal_subsec}
 \begin{wrapfigure}{r}{33mm}%
   \vspace{-10pt}
   \centering%
   \includegraphics[width=3cm]{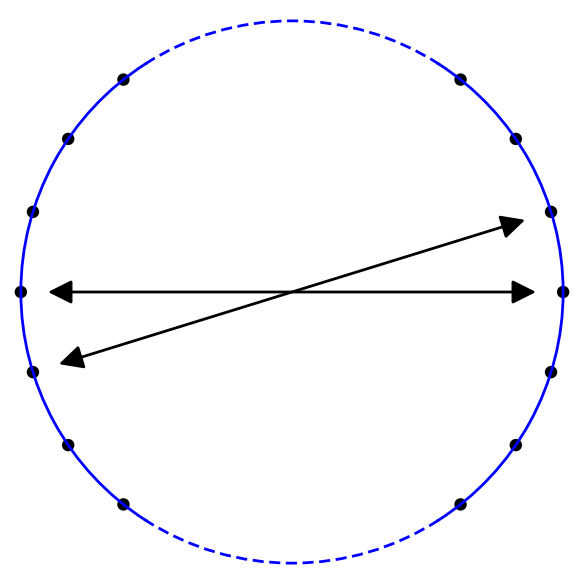}%
   \vspace{-30pt}%
 \end{wrapfigure}
 Let $E$ an elliptic curve over an algebraically closed field $\bar k$ of characteristic different from $p$.
 The following definition and lemma will be used in many places later on.

 \begin{defi}
   Let $\mathfrak v = (C_0, C_1,\dots, C_p)$ be a necklace in $E$. Two pearls $C_i$ and $C_j$ are called {\em antipodal in $\mathfrak v$} if $i\equiv j + \tfrac{p+1}{2}\pmod{p+1}$. In other words if they are diametrically opposed when we represent the necklace as a regular $(p+1)$-gon. If $A$ and $B$ are antipodal in $\mathfrak v$, we write $A\antip B \in \mathfrak v$.
 \end{defi}

 \begin{lem}\label{antip_pairing_lem}
  Let $A$, $B$, $C$, $D$ be four distinct cyclic subgroups of order $p$ in an elliptic curve $E$. There are $(p-1)/2$ necklaces in which $A\antip B.$ If the cross-ratio $[A,B;C,D]$ is a square in $\Fp^{\times}$, then there is no necklace $\mathfrak v$ such that $A\antip B \in \mathfrak v$ and $C\antip D\in \mathfrak v$. If instead $[A,B;C,D]$ is a non-square in $\Fp^{\times}$, then there is exactly one necklace $\mathfrak v$ such that $A\antip B \in \mathfrak v$ and $C\antip D\in \mathfrak v$.
 \end{lem}

 \begin{proof}
   We may choose a basis of $E[p]$ identifying $\PP(E[p])$ with $\PP^1(\Fp)$  in such a way that $A=\infty$, $B=0$ and $C=1$. Then $D=d$ for some $d\in \Fp^{\times}\setminus\{1\}$. Now $[A,B;C,D] = d$.

   If $d$ is a non-square, then the matrix $g=[\ma{0}{d}{1}{0}]$ is an element of order two without a fixed point in $\PP^1(\Fp)$. Hence it belongs to a unique non-split Cartan subgroup $\mathcal H$. Then the necklace $\mathfrak v$ whose stabiliser is the normaliser of $\mathcal H$ is a necklace such that $A\antip B\in\mathfrak v$ and $C\antip D\in\mathfrak v$ as $g(A) = B$ and $g(C)=D$.

   Conversely, if we have such a necklace $\mathfrak v$ for $A$, $B$, $C$, $D$, then the unique element of order $2$ which preserves the orientation on $\mathfrak v$, must send $A$ to $B$ and $C$ to $D$. Hence it is of the form $g=[\ma{0}{d}{1}{0}]$. However if it has no fixed points in $\PP^1(\Fp)$, then $d$ has to be a non-square in $\Fp^{\times}$.

   Finally, we have to count how many necklaces have $A\antip B\in \mathfrak v$. By the above proof, this is the same as to count how many matrices $g=[\ma{0}{d}{1}{0}]$ belong to a non-split Cartan subgroup. That is $\tfrac{p-1}{2}$ as there are that many non-squares $d$ in $\Fp^{\times}$.
 \end{proof}


\subsection{Elliptic points}\label{elliptic_points_subsec}

 We proceed to count elliptic points using our moduli description. Our results in Propositions~\ref{elliptic_points_prop} and~\ref{elliptic_points_plus_prop} below agree with the more general calculations by Baran in Proposition~7.10 in \cite{baran}. Assume for this that $p>3$.

 Consider the canonical coverings $\Xn\longrightarrow X(1)$ and $\Xn^+\longrightarrow X(1)$. An \emph{elliptic point} on $\Xn$ or $\Xn^+$ is a point in the fibre of a point in $X(1)$ represented by an elliptic curve $E$ with $\Aut(E)\neq\{\pm 1\}$. Hence, an elliptic point on $\Xn$ can be represented by a pair $(E,\mathfrak v)$ such that there is an automorphism on $E$ that induces a non-trivial element $g\in \PGL\bigl(E[p]\bigr)$ which fixes $\mathfrak v$. Consider the involution $w$ on $\Xn$ which reverses the orientation of the oriented necklaces. An elliptic point on $\Xn^{+}$ can be viewed as a pair $(E,\{\mathfrak v,w\mathfrak v\})$ with an automorphism $g \in \PGL\bigl(E[p]\bigr)$ and an oriented necklace $\mathfrak v$ such that either $g(\mathfrak v) = \mathfrak v$ or $g(\mathfrak v) = w(\mathfrak v)$. In the latter case, we say that $\mathfrak v$ and its necklace $\{\mathfrak v, w(\mathfrak v)\}$ is \textit{flipped} by $g$.

 First note that if $(E,\cdot)$ is an elliptic point then $j(E)=1728$ and $g$ is of order two or $j(E)=0$ and $g$ is of order three. These are elliptic curves with complex multiplication and $E[p]$ becomes a free $\End(E)/p\End(E)$-module of rank $1$. So if $g$ is of order $2$ and $p\equiv 3\pmod{4}$ or if $g$ is of order $3$ and $p\equiv 2 \pmod{3}$, then $\End(E)/p\End(E) \cong \Fps$ and hence $g$ belongs to a unique non-split Cartan subgroup of $\PGL\bigr(E[p]\bigl)$. Instead, if $g$ is of order $2$ and $p\equiv 1 \pmod{4}$ or if $g$ is of order $3$ and $p\equiv 1 \pmod {3}$ then $\End(E)/p\End(E)\cong \Fp\oplus \Fp$ and therefore $g$ belongs to a unique split Cartan subgroup as it will have exactly two fixed points.

\subsubsection{Fixed oriented necklaces}

 Let $(E,\mathfrak v)$ be an elliptic point on $\Xn$  with the oriented necklace $\mathfrak v$ fixed by $g$. Then $g$ is in the non-split Cartan subgroup stabilising $\mathfrak v$. Hence by the above, $p\equiv 3\pmod{4}$ if $g$ has order $2$ and $p\equiv 2 \pmod{3}$ if $g$ has order $3$. Conversely, if these congruence conditions are satisfied then $g$ is in a unique non-split Cartan subgroup which is the stabiliser of exactly two oriented necklaces, namely $\mathfrak v$ and $w\mathfrak v$. This gives the following result:

 \begin{prop}\label{elliptic_points_prop}
   For $r=2$ and $3$, let $e_r$ be the number of elliptic points in $\Xn$ with $g$ of order $r$. Then
   \begin{equation*}
      e_2 = 1 - \leg{-1}{p} = \begin{cases}
                                0 & \text{ if } p\equiv 1\pmod 4,\\
                                2 & \text{ if } p\equiv 3\pmod 4,
                              \end{cases}
      \quad \text{ and }\quad
      e_3 = 1-\leg{-3}{p} = \begin{cases}
                              0 & \text{ if } p\equiv 1\pmod 3,\\
                              2 & \text{ if } p\equiv 2\pmod 3.
                            \end{cases}
   \end{equation*}
   In the cases where $e_r=2$ the two corresponding oriented necklaces are in the same $w$-orbit.
 \end{prop}

\subsubsection{Flipped necklaces}

 Let $(E,\mathfrak v)$ be an elliptic point on $\Xn^+$ where the  necklace $\mathfrak v=\{\vec{\mathfrak v},\,w\vec{\mathfrak v}\}$ is flipped by $g$.  If $g$ were of order $3$, we would have $\vec{\mathfrak v}=g^3(\vec{\mathfrak v})=w(\vec{\mathfrak v})$. Hence $g$ is of order $2$. The involution $g$ is in a split Cartan subgroup if $p\equiv 1\pmod 4$ and in a non-split Cartan subgroup if $p\equiv 3\pmod 4$.

 \begin{lem}\label{flipped_one_lem}
   Suppose that $p\equiv 1\pmod 4$ and let $A,A'$ denote the two fixed points of $g$ in $\PP(E[p])$. A necklace $\mathfrak v$ is flipped by $g$ if and only if $A$ and $A'$ are antipodal in $\mathfrak v$. Consequently, there are $\tfrac{p-1}{2}$ necklaces flipped by $g$.
 \end{lem}

 \begin{wrapfigure}{r}{53mm}%
   \vspace{-10pt}
   \centering%
   \includegraphics[width=5cm]{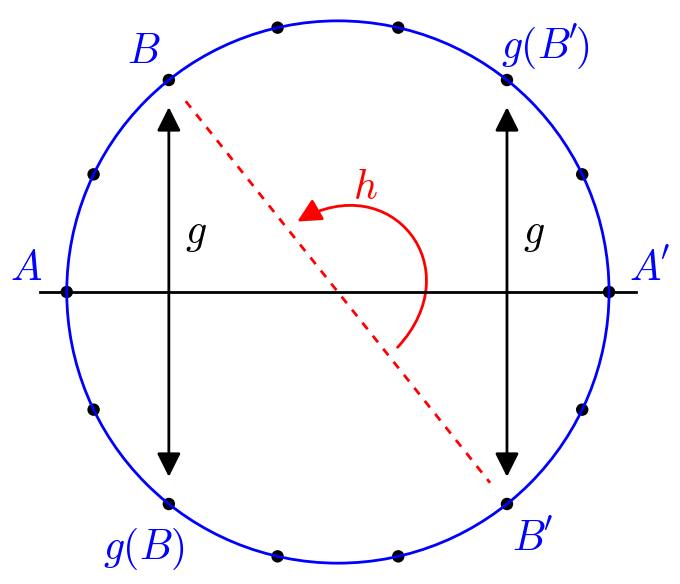}%
   \vspace{-10pt}%
 \end{wrapfigure}
  \noindent\textit{Proof.}
   Let $\vec{\mathfrak v}=(C_0,C_1,C_2,\dots,C_p)$ be a flipped oriented necklace with $C_0=A$. From $g(\vec{\mathfrak v})=w(\vec{\mathfrak v})$, we get $g(C_k)=C_{p+1-k}$ for all $k$, where the indices are taken modulo $p+1$. It follows that if $A'=C_k$ then $A'= g(A')=C_{p+1-k}$, so $k=(p+1)/2$ and $A$ and $ A'$ are antipodals in $ \mathfrak v=\{\vec{\mathfrak v}, w\vec{\mathfrak v}\}$.  Moreover from $g(C_k)=C_{p+1-k}$, we see that $g$ will act on  $\mathfrak v$, represented as a regular $(p+1)$-gon, as the reflection through the axis passing through $A$ and $A'$.

   Conversely, let $\mathfrak v$ be a necklace in which $A\antip A'$. Let $B\antip B'$ be two other antipodal  pearls in $\mathfrak v$. Let $h$ be the element of order $2$ in the normaliser of the non-split Cartan subgroup stabilising $\mathfrak v$. As it exchanges antipodal pairs in $\mathfrak v$, we have $h(A)=A'$ and $h(B)=B'$. Since $hgh^{-1}$ is also an involution that fixes $A$ and $A'$, it follows that $hgh^{-1} = g$ as there is a unique involution fixing two given points. Therefore $hg(B) = gh(B) = g(B')$ which implies that $g(B)\antip g(B')\in\mathfrak v$.

   As $g$ sends antipodal pairs in $\mathfrak v$ to antipodal pairs in $g(\mathfrak v)$, we also have $A\antip A'$  and $g(B)\antip g(B')$ in $g(\mathfrak v)$. Hence, by Lemma~\ref{antip_pairing_lem}, either $g(\vec{\mathfrak v})=\vec{\mathfrak v}$ or $g(\vec{\mathfrak v})=w\vec{\mathfrak v}$ where $\{\vec{\mathfrak v},w\vec{\mathfrak v}\}=\mathfrak v$. The first case is excluded because $g$ does not belong to a non-split Cartan subgroup if $p\equiv 1\pmod 4$.

   The end of the proof follows from the fact that there are $(p-1)/2$ necklaces such that $A$ and $A'$ are antipodal, again by Lemma~\ref{antip_pairing_lem}.
  \hfill\qedsymbol

 \begin{lem}\label{flipped_two_lem}
   Suppose that $p\equiv 3\pmod 4$. Let $A\in\PP\bigl(E[p]\bigr)$. Consider the map sending a necklace $\mathfrak v$ to the pearl antipodal to $A$ in $\mathfrak v$. This is a bijection between necklaces $\mathfrak v$ flipped by $g$ and the set of pearls $B\not\in \{A,g(A)\}$ such that $\bigl[A,B;g(A),g(B)\bigr]$ is a non-square in $\Fp$.
 \end{lem}

 \begin{wrapfigure}{r}{53mm}%
   \vspace{-10pt}
   \centering%
   \includegraphics[width=5cm]{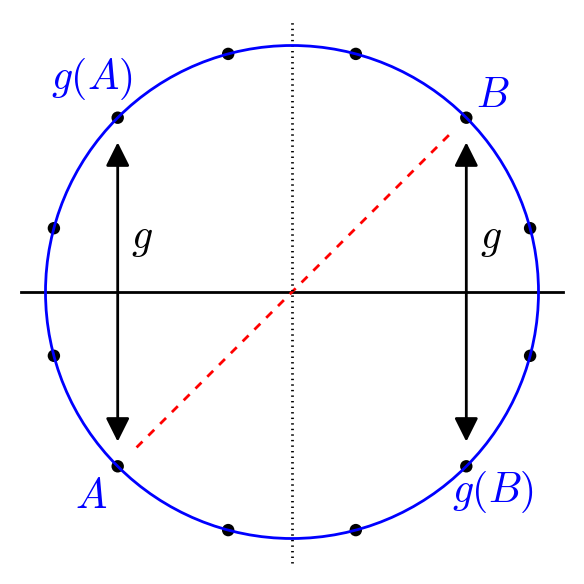}%
   \vspace{-10pt}%
 \end{wrapfigure}
  \noindent\textit{Proof.}
   Let $\mathfrak v$ be a necklace flipped by $g$. As above, from $g(\vec{\mathfrak v})=w(\vec{\mathfrak v})$ we get  $g(C_k)=C_{p+1-k}$ for all $k$ and one can see that, since $p\equiv 3\pmod 4$,  $g$ will act on $\mathfrak v$ as a reflection through an axis that does not pass through a corner of the regular $(p+1)$-gon. Let $B$ be antipodal to $A$ in $\mathfrak v$. So $B\neq A$. If $B$ were equal  to $g(A)$, then $A$ and $B$ would be on the line orthogonal to the axis of reflection of $g$. But this would imply that $p+1\equiv 2 \pmod 4$, and hence $B\neq g(A)$. Finally, since $g$ flips $\mathfrak v$, we see that $g(A) \antip g(B)\in\mathfrak v$. By Lemma~\ref{antip_pairing_lem}, it follows that $\bigl[A,B;g(A),g(B)\bigr]$ is a non-square modulo $p$. The same lemma also shows that our map $\mathfrak v\mapsto B$ is injective.

   Conversely, suppose that $B\not\in\{A,g(A)\}$ is such that the cross-ratio $\bigl[A,B;g(A),g(B)\bigr]$ is a non-square modulo $p$. Since $A$, $B$, $g(A)$ and $g(B)$ are all distinct, Lemma~\ref{antip_pairing_lem} applies to show that there is a necklace $\mathfrak v$ with $A \antip B$ and $g(A)\antip g(B)$. Now $g(\mathfrak v)$ has also $g(A) \antip g(B)$ and $A\antip B$. The same lemma now shows that $g(\mathfrak v) = \mathfrak v$. If the orientation of $\mathfrak v$ were fixed rather than flipped, then $g(A)$ would be $B$. Hence our map is surjective, too.
 \hfill\qedsymbol

 \begin{prop}\label{elliptic_points_plus_prop}
   For $r=2$ or $3$, let $e_r^+$ be the number of elliptic points with $g$ of order $r$ in $\Xn^+$. Then
   \begin{equation*}
     e_2^+ = \frac{p+1}{2}  -\leg{-1}{p} =
     \begin{cases}
        \frac{p-1}{2} & \text{ if } p\equiv 1\pmod 4,\\
        \frac{p+3}{2} & \text{ if } p\equiv 3\pmod 4,
     \end{cases}
   \end{equation*}
   and
   \begin{equation*}
     e_3^+ = \frac{1}{2} -\frac{1}{2} \, \leg{-3}{p} =
     \begin{cases}
       0 & \text{ if } p\equiv 1 \pmod 3,\\
       1 & \text{ if } p\equiv 2\pmod 3.
     \end{cases}
   \end{equation*}
 \end{prop}

 \begin{proof}
   The number of elliptic points for $\Xn^{+}$ is the sum of the number of fixed and the number of flipped necklaces. In Proposition~\ref{elliptic_points_prop} we counted the fixed ones. We have already counted the flipped necklaces for $p\equiv 1 \pmod{4}$ in Lemma~\ref{flipped_one_lem}. Now suppose $p\equiv 3 \pmod{4}$ and let $A\in\PP\bigl(E[p]\bigr)$. By Lemma~\ref{flipped_two_lem}, we must count how many pearls $B$ there are such that $B\not\in \{A,g(A)\}$ and $\bigl[A,B;g(A),g(B)\bigr]$ is a non-square in $\Fp$.

   Let us choose a basis of $E[p]$ such that $A=\infty$ and $g(A)=0$. Let $b\in\Fp^\times$ such that $B=1/b$. Then $g=\bigl[\ma{0}{-1}{1}{0} \bigr] $ and $g(B)=-b$. Hence $\bigl[A,B; g(A),g(B)\bigr]=1+b^2$. So we have to count the number of $b\in\Fp^{\times}$ such that $1+b^2$ is a non-square. One finds that there are $\tfrac{p+1}{2}$ such $b$ by counting the cases when $1+b^2$ is a square using that there are $p+1$ points on a projective conic $a^2 + b^2 = c^2$.
 \end{proof}


\subsection{Genus}\label{genus_subsec}

 From the above, we can now proceed to compute the genus of our modular curves. Of course, we find the well-known formulae, as for instance in Appendix~A.5 to~\cite{serre}, \cite{baran} or \cite{chen1}. The reader can also find tables for the genus of $\Xn$ and $\Xn^+$ for small primes $p$ in \cite{baran}.

 The Riemann-Hurwitz formula applied to the modular curve $X_\mathcal H$ associated to a subgroup of finite index $\mathcal H$ of $\GL_2(\Fp)$ and with the canonical morphism $X_\mathcal H\to X(1)$ of degree $d$, gives the following formula for the genus $g(X_\mathcal H)$ of $X_\mathcal H$:
 \begin{equation*}
   g(X_\mathcal H)=1+\frac{d}{12}-\frac {e_2}{4}-\frac{e_3}{3}-\frac{e_\infty}{2}
 \end{equation*}
 where $e_r$ is the number of elliptic points in $X_H$ of order $r$ and $e_\infty$ is the number of cusps.

 With the results of Sections~\ref{elliptic_points_subsec} and~\ref{cusps_subsec}, a straightforward computation gives the following.

 \begin{prop}
   The genera of $\Xn$ and $\Xn^+$ are
   \begin{equation*}
     g(\Xn)=\frac{1}{12}\left(p^2-7p+11+3\leg{-1}{p}+4\leg{-3}{p}\right)
   \end{equation*}
   and
   \begin{equation*}
     g(\Xn^+)=\frac 1{24}\left(p^2-10p+23+6\leg{-1}{p}+4\leg{-3}{p}\right).
   \end{equation*}
 \end{prop}

 With the same method, one can compute the genus of other modular curves, for instance $X_0$ and $ \Xs^{+}$. The classical results (see for instance~\cite{shimura} and~\cite{chen1}) for their genus are
 \begin{equation*}
   g\bigl(\Xs^+\bigr) = \frac{1}{24} \left(p^2-8p+11-4\leg{-3}{p}\right)\
   \text{ and }\
   g\bigl(X_0\bigr) = \frac{1}{12} \left(p-6-3\leg{-1}{p}-4\leg{-3}p \right).
 \end{equation*}
 Then one can verify easily the relation noticed by Birch following Chen's calculation of genus and confirmed  by Chen's isogeny
 \begin{equation}\label{genus_eq}
   g\bigl(\Xn^+\bigr) + g\bigl(X_0\bigr) = g\bigl(\Xs^+\bigr).
 \end{equation}


\subsection{Hecke operators}

 Let $\ell$ be any prime distinct from $p$. Denote by $ \Xon^+(\ell,p)=X_0(\ell)\times_{X(1)}\Xn^+(p)$.  We recall how the Hecke correspondence $T_\ell$ is defined through the following two natural degeneracy maps $\rho$ and $\rho' \colon \Xon^+(\ell,p)\longrightarrow \Xn^+(p)$. The modular curve $\Xon^+(\ell,p)$ parametrises isomorphism classes $\bigl(E,(f,\mathfrak v)\bigr)$ of elliptic curves $E$ endowed with an $\ell$-isogeny $f\colon E \to E'$ and a necklace $\mathfrak v$. Let $\rho$ be the map obtained by forgetting the $\ell$-structure and $\rho'$ the map which sends $(E,(f,\mathfrak v))$ to $\bigl(E',f(\mathfrak v)\bigr)$. The image $f(\mathfrak v)$, defined as $\bigl(f(C_0), f(C_1), \dots, f(C_p)\bigr)$ when $\mathfrak v = (C_0,\dots, C_p)$, is indeed a necklace on $E' = f(E)$ since $\ell\neq p$.

 The correspondence $T_\ell$ on $\Xn^{+}$ is now defined as $\rho^*\circ \rho'_{*}$. It induces an endomorphism on $\Pic(\Xn^+)$ by Picard functoriality. On the divisor $(z)$ with the point $z$ in $\Xn^{+}$ represented by $(E,\mathfrak v)$, it is defined as
 \begin{equation*}
   T_{\ell}(z)=\sum_{\substack{f\colon E\to E' \\ \deg f = \ell}} \bigl(E', f(\mathfrak v) \bigr)
 \end{equation*}
 where the sum runs over all isogenies $f$ from $E$ of degree $\ell$.

 We will now verify that this moduli-theoretic description of $T_{\ell}$ via these correspondences coincide with the Hecke operators defined by double coset.

 Let us denote by $\mathcal N$ a normaliser of a non-split Cartan in $\GL_2(\Fp)$, by $\Gamma$ the congruence subgroup of matrices in $\SL_2(\ZZ)$ with image in $\mathcal N$ modulo $p$, and by $\Delta$ the set of integral matrices whose determinant is positive and coprime to $p$ and which reduce to a matrix in $\mathcal N$ modulo $p$. Then $T_\ell$ as a double coset is defined by $\Gamma\alpha\Gamma$ for any  $\alpha\in \Delta$ of determinant $\ell$.  Since such an element $\alpha$ reduces modulo $\ell$ to a non-zero matrix of determinant $0$, there exist $\gamma$ and $\gamma'$ in $\Gamma$ such that $\gamma\alpha\gamma'$ is of the form $\left(\begin{smallmatrix} a &b\ell\\ c\ell& d\ell \end{smallmatrix}\right)$ with integers $a$, $b$, $c$, and $d$. (Note that we can impose conditions modulo $p$ on $\gamma,\gamma'$ because $p$ and $\ell$ are coprime.) In other words, $\Gamma\alpha\Gamma$ contains an element of the form $\beta\left(\begin{smallmatrix} 1&0\\0&\ell \end{smallmatrix}\right)$ with $\beta \in \Gamma_0(\ell)$, so we may now suppose that $\alpha=\beta\left(\begin{smallmatrix} 1&0\\0&\ell \end{smallmatrix}\right)$, with $\beta\in\Gamma_0(\ell)$.

 We will then prove that  $\Gamma\cap \alpha^{-1}\Gamma\alpha = \Gamma\cap\Gamma^0(\ell)$, where we denote by $\Gamma^0(\ell)$ the matrices in $\SL_2(\ZZ)$ with right upper entry equal to $0$ modulo $\ell$: Since $\beta\in\SL_2(\ZZ)$, we have $\beta^{-1}\Gamma\beta \subset \SL_2(\ZZ)$, hence $ \bigl(\begin{smallmatrix} 1 & 0 \\ 0 & \ell \end{smallmatrix}\bigr)^{-1} \beta^{-1}\Gamma\beta \bigl(\begin{smallmatrix} 1 & 0 \\ 0 & \ell \end{smallmatrix}\bigr)$ are matrices of the form  $\bigr(\begin{smallmatrix} a & b\ell \\ c/\ell & d \end{smallmatrix}\bigr)$ and this shows the first inclusion. Conversely, if  $\gamma\in \Gamma\cap \Gamma^0(\ell)$, since $\alpha$ and $\gamma$ belong to $C$ modulo $p$, the product $\alpha\gamma\alpha^{-1}$ also belongs to it, so $\alpha\gamma\alpha^{-1}\in\Gamma$.

 Similarly, one can show that  $\Gamma\cap\alpha\Gamma\alpha^{-1} = \Gamma\cap\Gamma_0(\ell)$. Now, from  the fact that  $\Gamma\cap \alpha^{-1}\Gamma\alpha = \Gamma\cap\Gamma^0(\ell)$ or $\Gamma\cap\alpha\Gamma\alpha^{-1} = \Gamma\cap\Gamma_0(\ell)$, we  deduce in a classical manner, as explained for instance in Section~6.3 in~\cite{diamond_shurman} taking $\Gamma_1=\Gamma_2=\Gamma$, that the double coset description of $T_\ell$ coincide with the moduli-theoretic description we gave above.
 Compare with Theorem~1.11 in~\cite{kohen_pacetti} for a different proof.


\subsection{A pairing}\label{pairing_subsec}

 Given two necklaces $\mathfrak v$ and $\mathfrak w$ in $E$, we set
 \begin{equation*}
   \langle \mathfrak v,\mathfrak w\rangle = \# \Bigl\{ \{A,B\}\ \Bigm\vert\ A \antip B \in \mathfrak v \text{ and } A \antip B \in \mathfrak w \Bigr\}.
 \end{equation*}
 It is the number of antipodal pearls that $\mathfrak v$ and $\mathfrak w$ have in common. We can extend it linearly to $\bigoplus_{\text{all }\mathfrak v} \ZZ \mathfrak v$ regarded as an abelian group with an action by $\PGL\bigl(E[p]\bigr)$.

 \begin{prop}\label{pairing_prop}
   The pairing $\langle\cdot,\cdot\rangle$ is a positive non-degenerate symmetric $\PGL\bigl(E[p]\big)$-equivariant bilinear form on $\bigoplus_{\mathfrak v} \ZZ \mathfrak v$. We have $\langle \mathfrak v , \mathfrak v\rangle = \tfrac{p+1}{2}$ and $\langle \mathfrak v, \mathfrak w \rangle \in\{0, 1\}$ for all necklaces $\mathfrak v\neq \mathfrak w$.
 \end{prop}

 First, we note that we are left to prove that the pairing is positive, takes value $0$ or $1$ on distinct necklaces, and is non-degenerate. In this section, we only give the proof of the two first facts. The proof of non-degeneracy will be given in Section~\ref{pairing_proof_subsec} and numerical examples are in Section~\ref{examples_sec}. We will see that this pairing gives a more conceptual understanding of the eigenvalues computed by Chen~\cite{chen2} in his table~2.

 \begin{proof}
   The statement that $\langle \mathfrak v,\mathfrak w\rangle \in \{0,1\}$ for $\mathfrak v\neq \mathfrak w$ is a direct consequence of Lemma~\ref{antip_pairing_lem}: If $\langle \mathfrak v,\mathfrak w\rangle \geq 2$, then there are four distinct $A$, $B$, $C$, $D$ with both $A\antip B$ and $C\antip D$ in $\mathfrak v$ and $\mathfrak w$, contradicting the lemma.

   Let $u = \sum a_{\mathfrak v}\, \mathfrak v$ be an element in $\bigoplus \ZZ\, \mathfrak v$. We have
   \begin{equation*}
      \langle u,u\rangle = \sum_{\mathfrak v}\sum_{\mathfrak w} a_{\mathfrak v} \,a_{\mathfrak w} \,\langle \mathfrak v,\mathfrak w\rangle =
      \sum_{ \{A,B\} } \sum_{\substack{\mathfrak v\text{ with}\\ A\antip B\in \mathfrak v}}  \sum_{\substack{\mathfrak w\text{ with}\\ A\antip B\in \mathfrak w}} a_{\mathfrak v}\, a_{\mathfrak w}
      = \sum_{ \{A,B\} }\Bigl( \sum_{\substack{\mathfrak v\text{ with}\\ A\antip B\in \mathfrak v}} a_{\mathfrak v} \Bigr)^2 \geq 0,
   \end{equation*}
   where $\sum_{\{A,B\}}$ is the sum running over all unordered pairs of distinct cyclic subgroups of $E[p]$. Hence   the pairing is positive. The non-degeneracy of the pairing will be shown in Section~\ref{pairing_proof_subsec}.
 \end{proof}


\section{Chen's isogeny}\label{chen_sec}

\subsection{Definitions and statement}\label{chen_def_subsec}

 In~\cite{chen1}, Chen proved that $\Jac(\Xn^{+}) = \Jac(X_0^{+}(p^2))^{\text{new}}$. Edixhoven and de~Smit~\cite{desmit_edixhoven,edixhoven} found a different and rather elegant proof. Finally Chen gave in~\cite{chen2} an explicit description of his morphism
 \begin{equation*}
    \Jac(\Xs^{+} ) \to \Jac(\Xn^{+} )\times\Jac(X_0).
 \end{equation*}
 With our new moduli description this morphism can be described yet in another manner. Let $\bar k$ be an algebraically closed field of characteristic different from $p$. In Section~\ref{genus_subsec}, we have given the definitions of the modular curves $X_0$ and $\Xs^{+}$. The points in $\Xs^{+}(\bar k)$ can be represented as $\bar k$-isomorphism classes of the form $\bigl(E,\{A,B\}\bigr)$ where $\{A,B\}$ is a unordered pair of distinct cyclic subgroups of order $p$ in $E$. Let $y = \bigl(E, \{A, B\}\bigr)$ be a point in $\Xs^{+}$. We define $\varphi(y)$ on the divisor $(y)$ to be the sum of $(E,\mathfrak{v})$ where $\mathfrak{v}$ runs over all necklaces in which the pearls $A$ and $B$ are antipodal. This extends linearly to a map on the Jacobians.

 Further we define the maps $\psi$, $\mu$, and $\lambda$ in the diagram on the right as follows. If $y$ is the point $\bigl(E, \{A,B\}\bigr)$ as above, then $\mu(y)$ is the sum of the points $(E,A)$ and $(E,B)$ in $X_0$. If $x=(E,A)$ with $A$ a cyclic subgroup of order $p$ in $E$ is a point on $X_0$, then we set $\psi(x)$ equal to the sum of $\bigl(E, \{A,B\}\bigr)$ where $B$ runs through all cyclic subgroups of order $p$ in $E$ distinct from $A$. Finally if $z=(E, \mathfrak{v})$ is a point on $\Xn$ for some necklace $\mathfrak v$, then $\lambda(z)$ is the sum of all $\bigl(E, \{A,B\}\bigr)$ where $A$ and $B$ run through all pairs of antipodal pearls in the necklace $\mathfrak v$. All these three correspondences extend linearly to the corresponding Jacobians. As explained in Section~\ref{chen_grth_subsec}, those correspondences comes from degeneracy maps (where we replace $\pin^+$ by $\piin$ which will be defined in Section~\ref{chen_grth_subsec}).

 To recapitulate, we list in Figure~\ref{defs_fig} the definitions for future reference next to a diagram involving all relevant maps.

 \begin{figure}[hb]
  \begin{minipage}{7cm}
     \begin{equation*}
       \xymatrix@C+1em@R+1em{
         & X(p) \ar[d] & \\
         & X_{\mathcal A} \ar[dl]_{\pi_0} \ar[d]^{\pis^{+}} \ar[dr]^{\piin} & \\
         X_0 \ar@<2pt>@{..>}^{\psi}[r] \ar[dr]_{j_0} & \Xs^{+} \ar@<2pt>@{..>}^{\varphi}[r] \ar@<2pt>@{..>}^{\mu}[l]\ar[d]^{\js^{+}} & \Xn^{+} \ar[dl]^{\jn^{+}}  \ar@<2pt>@{..>}^{\lambda}[l] \\
         & X(1)\cong \PP^1 &
       }
     \end{equation*}
   \end{minipage}
 \begin{minipage}{\textwidth-7.5cm}
  \begin{align*}
   \psi\bigl( E, A\bigr) &= \sum_{\substack{B\text{ with}\\B\neq A}} \bigl(E, \{A,B\}\bigr) \\
   \mu\bigl(E, \{A,B\}\bigr) &= \bigl(E,A\bigr) + \bigl(E,B\bigr) \\
   \varphi\bigl( E, \{A,B\}\bigr) &= \sum_{\substack{\mathfrak v\text{ with}\\\mathfrak v \ni A \antip B}} \bigl(E,\mathfrak v\bigr) \\
   \lambda\bigl( E, \mathfrak v \bigr) &= \sum_{\substack{\{A,B\}\text{ with}\\A\antip B\in \mathfrak v}} \bigl(E,\{A,B\}\bigr)
  \end{align*}
 \end{minipage}
  \caption{The various maps in Chen's theorem}\label{defs_fig}
 \end{figure}
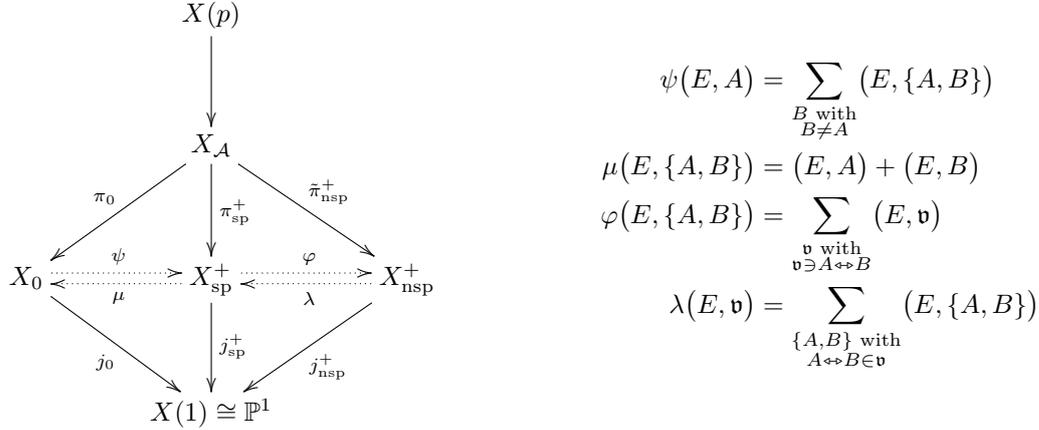

 We proceed to give a new proof of  Chen's result. Even if we believe that our proof is simpler and conceptually better visualised than the original proof in~\cite{chen2}, we have to emphasise that it is mostly a reformulation or translation of Chen's proof into our new language: As  we explore it in detail in Subsection~\ref{chen_grth_subsec}, the main argument is of the same nature as in Chen's proof.

 \begin{thm}[Chen-Edixhoven]\label{chen_thm}
   There are two complexes of abelian varieties over $\QQ$
   \begin{equation*}
     \xymatrix@R-7mm{
       0 \ar[r] & \Jac\bigl(X_0\bigr)\ar^{\psi}[r] & \Jac\bigl(\Xs^{+}\bigr)\ar^{\varphi}[r] & \Jac\bigl(\Xn^{+}\bigr) \ar[r] & 0 \\
       0 & \ar[l] \Jac\bigl(X_0\bigr) &  \ar^{\mu}[l] \Jac\bigl(\Xs^{+}\bigr) & \ar^{\lambda}[l] \Jac\bigl(\Xn^{+}\bigr) & \ar[l] 0
     }
   \end{equation*}
   whose cohomologies are finite groups.
 \end{thm}

 We could also reformulate the theorem by saying that
 \begin{equation*}
   \xymatrix@1@C+1cm{
     \Jac\bigl(X_0\bigr)\oplus \Jac\bigl(\Xn^{+}\bigr)\
     \ar@<2pt>^(0.6){\psi +\lambda}[r] &
     \ar@<2pt>^(0.4){\mu\oplus\varphi}[l]\
     \Jac\bigl(\Xs^{+}\bigr)
   }
 \end{equation*}
 are isogenies defined over $\QQ$; however they are not dual to each other.


\subsection{The easier part of the proof}

 \begin{lem}
   The two sequences in Theorem~\ref{chen_thm} are complexes and $\mu\circ\psi = [p-1]$ on the Jacobian of $X_0$.
 \end{lem}

 \begin{proof}
   Let $x=(E,A)$ be a point in $X_0$. Then
   \begin{equation*}
     \varphi\circ \psi (x) = \varphi\Bigl(\sum_{B\neq A} \bigl(E,\{A,B\}\bigr)\Bigr) = \sum_{B\neq A}\  \sum_{\mathfrak v \ni A \antip B} (E,\mathfrak v)
   \end{equation*}
   where the last sum runs over all necklaces $\mathfrak v$ in which $A$ and $B$ are antipodal. Now $A$ will appear in each necklace once and for each necklace there is a unique $B$ which is antipodal to $A$ in $\mathfrak v$. Hence $\varphi\circ \psi (x)$ is equal to the sum over all possible necklaces of $E$. But
   \begin{equation*}
      \varphi\circ \psi (x) = \sum_{\mathfrak v} (E,\mathfrak v) = \jn^{*}(E)
   \end{equation*}
   is the pullback of a divisor $(E)$ on $X(1)$ by the natural projection $\jn^{+}\colon \Xn^{+}\to X(1)$. Since $X(1)\cong \PP^1$ has trivial Jacobian, we find $\varphi\circ\psi=0$ on the Jacobians. This proof was already noted by Merel on page~189 in~\cite{merel}.

   Next, for a point $z=(E,\mathfrak v)$ in $\Xn^{+}$, we have
   \begin{equation*}
     \mu\circ\lambda(z) = \mu\Bigl( \tfrac{1}{2} \sum_A \bigl(E,\{A,B\}\bigr) \Bigr)= \frac{1}{2} \sum_A \Bigl( (E,A) + (E,B)\Bigr) =   \sum_A (E,A) =   j_0^{*}(E)
   \end{equation*}
   where $B$ in the sums denotes the unique pearl which is antipodal to $A$ in $\mathfrak v$ and where $j_0\colon X_0\to X(1)$. Hence $\mu\circ\lambda=0$ on the Jacobians.

   Finally, we obtain
   \begin{align*}
     \mu\circ\psi(x) &= \mu\Bigl(\sum_{B\neq A} \bigl(E,\{A,B\}\bigr) \Bigr) = \sum_{B\neq A} \Bigl( (E,A) + ( E,B)\Bigr) \\
     &= (p-1)\cdot (E,A) + \sum_{B} (E,B) = (p-1)\cdot x + j_0^{*}(E)
   \end{align*}
   and hence $\mu\circ\psi=[p-1]$ on the Jacobian of $X_0$.
 \end{proof}

 \begin{cor}
   The kernel $\ker\psi\subset \Jac(X_0)[p-1]$ and the cokernel $\coker(\mu)=0$ are finite.
 \end{cor}


\subsection{Making use of antipodal pearls}

 We deduce from the earlier Lemma~\ref{antip_pairing_lem} the following result:

 \begin{cor}\label{laphi_cor}
   For every $\bigl(E,\{A,B\}\bigr)\in \Xs^+$, we have
   \begin{equation*}
     \lambda\circ\varphi\,\bigl(E,\{A,B\}\bigr) = \frac{p-1}{2}\cdot \bigl(E, \{A,B\}\bigr) + \sum_{\substack{\{C,D\}\text{ with}\\ [A,B;C,D] \not\in\square}} \bigl(E, \{C, D\}\bigr)
   \end{equation*}
   with the sum running over all $\{C,D\}$ disjoint from $\{A,B\}$ such that the cross-ratio $[A,B;C,D]$ is a non-square in $\Fp^{\times}$.
 \end{cor}

 \begin{proof}
   Since
   \begin{equation*}
     \lambda\circ\varphi\, \bigl(E,\{A,B\}\bigr) = \sum_{\substack{\mathfrak v\text{ with}\\ A\antip B\in \mathfrak v}}\ \sum_{\substack{\{C,D\}\text{ with}\\C\antip D\in \mathfrak v}} \bigl(E, \{C, D\}\bigr)
   \end{equation*}
   we are asked to count how many necklaces have both $\{A,B\}$ and $\{C,D\}$ as antipodal pairs in common. If the four pearls are distinct, Lemma~\ref{antip_pairing_lem} gives the answer. If $A=B$, but $C\neq D$, then there are no such $\mathfrak v$ and if $\{A,B\}=\{C,D\}$, then we have to count how many necklaces have $A\antip B\in \mathfrak v$, this  is $\tfrac{p-1}{2}$ by Lemma~\ref{antip_pairing_lem} again.
 \end{proof}

 We now define yet another map $\alpha\colon \Jac\bigl(\Xs^+\bigr) \to \Jac\bigl(\Xs^+\bigr)$. For a point $y=(E,\{A,B\})$, we define $\alpha\bigl(E,\{A,B\}\bigr)$ to be the sum $\sum \{C,D\}$ running over all unordered pairs $\{C,D\}$ such that $[A,B;C,D]=-1$.

 \begin{lem}\label{alphaalpha_lem}
   We have
   \begin{equation*}
     \alpha\circ\alpha\, \bigl( E,\{A,B\} \bigr) = \frac{p-1}{2}\cdot \bigl(E, \{A, B\}\bigr) + \sum_{\substack{\{C,D\}\text{ with}\\ {} [A,B;C,D]\in\square}} \bigl(E, \{C, D\}\bigr)
   \end{equation*}
   where the second sum runs over all unordered pairs $\{C,D\}$ such that the cross-ratio $[A,B;C,D]$ is a square in $\Fp^{\times}$.
 \end{lem}

 \begin{proof}
   By definition, we have
   \begin{equation*}
     \alpha\circ\alpha \, \bigl( E,\{A,B\} \bigr) =
     \sum_{ \substack{ \{X,Y\}\text{ with}\\ {} [A,B;X,Y] = -1}} \
     \sum_{ \substack{ \{C,D\}\text{ with}\\ {} [X,Y;C,D] = -1}} \{C,D\}.
   \end{equation*}
   Given $\{C,D\}$, we wish to determine how many $\{X,Y\}$ exist with $[A,B;X,Y]=[X,Y;C,D]=-1$. Assume first that $A$, $B$, $C$, $D$ are all distinct. It follows that $X$ and $Y$ are distinct from any of the four. Then we choose a basis, identifying $\PP(E[p])$ with $\PP^1(\Fp)$, such that $A=\infty$, $B=0$ and $C=1$. Through this identification, we write $D=d$, $X=x$ and $Y=y$. The two equations give
   \begin{align*}
     -1 &= [A,B;X,Y] = x/y \\
     -1 &= [X,Y;C,D] = \frac{x-1}{y-1}\cdot \frac{y-d}{x-d}.
   \end{align*}
   They simplify to $x=-y$ and $x^2=d=[A,B;C,D]$. Hence if $[A,B;C,D]$ is a non-square in $\Fp$, then there are no $\{X,Y\}$ and if it is a square then there is exactly one pair $\{X,Y\}$.

   Finally, suppose they are not all distinct, say $C=A$. If $D\neq B$, then there can not be any $\{X,Y\}$. If $\{A,B\}=\{C,D\}$, then all pairs $\{X,Y\}$ with $[A,B;X,Y]=-1$ will contribute to the sum, and there are $\tfrac{p-1}{2}$ such pairs.
 \end{proof}

 \begin{prop}
   Let $\js\colon \Xs^+\to X(1)$ be the natural projection. The relation
   \begin{equation}\label{chen86_eq}
     \Bigl(\lambda\circ \varphi + \alpha\circ\alpha + \psi\circ\mu\Bigr) \bigl(E,\{A,B\}\bigr) = p\cdot \bigl(E, \{A,B\}\bigr) + \js^*(E)
   \end{equation}
   holds for all $\bigl(E,\{A,B\}\bigr)\in \Xs^+$.
 \end{prop}

 \begin{proof}
   This is just the combination of Corollary~\ref{laphi_cor}, Lemma~\ref{alphaalpha_lem}, the equality
   \begin{equation*}
     \psi\circ\mu\bigl(E,\{A,B\}  \bigr) = 2\cdot \bigl(E,\{A,B\}\bigr) + \sum_{C\neq A,B} \Bigr( \bigl(E,\{A,C\}\bigr) + \bigl(E,\{B,C\}\bigr)\Bigl),
   \end{equation*}
   and counting how often $\bigl(E,\{A,B\}\bigr)$ appears on both sides.
 \end{proof}


\subsection{Representation theoretic argument}\label{grth_subsec}

 The main argument in~\cite{edixhoven,desmit_edixhoven} that an isogeny must exist between the Jacobians, and even some information about its degree, is directly deduced from the Brauer relation between certain permutation representation. Denote by $\mathcal{B}$, $\mathcal S$, and $\mathcal N$ a Borel subgroup, a normaliser of a split Cartan subgroup and a normaliser of a non-split Cartan subgroup of a group $\mathcal G$ isomorphic to $\PGL_2(\Fp)$, respectively. Then (see~\cite{edixhoven,desmit_edixhoven})
 \begin{equation*}
   \QQ[\mathcal G/\mathcal S] \oplus \QQ[\mathcal G/\mathcal G] \cong \QQ[\mathcal G/\mathcal N]\oplus \QQ[\mathcal G/\mathcal B].
 \end{equation*}

 We fix an elliptic curve $E$ over an algebraically closed field of characteristic different from $p$. Write $\mathcal G=\PGL\bigl(E[p]\bigr)$. We consider the $\QQ[\mathcal G]$-modules $U=\bigoplus_{\mathfrak v}\QQ\, (E,\mathfrak v)$, which is isomorphic to $\QQ[\mathcal G/\mathcal N]$, and $V = \bigoplus_{\{A,B\}} \QQ \,\bigl(E, \{A,B\}\bigr)\cong\QQ[\mathcal G/\mathcal S]$. The equation~\eqref{chen86_eq} is a relation between $\QQ[\mathcal G]$-endomorphisms of $V$: $\lambda\circ\varphi+\alpha\circ\alpha+\psi\circ\mu=[p]+j$, where we still denote by $\psi,\varphi,\alpha,\lambda,\mu$ the morphisms induced on $\QQ[\mathcal G]$-modules and where $j\colon V \to V$ sends $\bigl(E,\{A,B\}\bigr)$ to the sum over all $\bigl(E,\{C,D\}\bigr)$.

 From the fact that the middle line (for $T'$) in table~2.2 in~\cite{edixhoven} only contains $0$ and $1$, we see that $V\otimes\CC$ decomposes into a sum of \textit{distinct} irreducible $\CC[\mathcal G]$-modules.  We denote by $\chi_W$ the character and $e_W = 1/\vert \mathcal{G} \vert \cdot  \sum_{g\in \mathcal{G}} \chi_W(g)\, g^{-1}$ the  idempotent associated to such an irreducible $\CC[\mathcal G]$-submodule $W$ of $V$. For $f$  a $\QQ[\mathcal G]$-endomorphism of $V$, Schur's Lemma implies that $f|_W$  is the multiplication by a scalar $c_W(f)\in\CC$. Let $K$ be the cyclotomic field $\QQ(\zeta_{p-1},\zeta_{p+1})$. A look at the character table (for instance table~2.1 in~\cite{edixhoven}) of $\mathcal G \cong \PGL_2(\Fp)$ shows that all values of characters are contained in $K$.  Since $c_W(f)=1/\dim(W)\cdot \tr(e_W\circ f)$, we see that $c_W(f)$ belongs to $\QQ(\chi_W) \subset K$.

 We will now consider these scaling factors for the $\QQ[\mathcal G]$-endomorphisms in equation~\eqref{chen86_eq}. Let $W$ be an irreducible complex representation which appears in the decomposition of $V\otimes \CC$ but not in the image of $\psi\otimes \CC\colon \bigoplus_A\CC\, (E, A)\to V\otimes \CC$. Then $c_W(\psi\circ \mu)=0$. By the Brauer relation, since $\bigoplus_A\QQ\, (E, A)\cong\QQ[\mathcal G/\mathcal B]$, the representation $W$ also appears in the decomposition of $U\otimes \CC$. Then similarly $c_W(j)=0$.  Hence the equation~\eqref{chen86_eq} gives
 \begin{equation*}
   c_W(\lambda\circ\varphi) = c_W([p]) - c_W(\alpha\circ\alpha) = p - c_W(\alpha)^2.
 \end{equation*}
 However, since $c_W(\alpha)\in K$, it can not be equal to $\pm\sqrt{p}$. This shows that $c_W(\lambda\circ \varphi)\neq 0$ for all irreducible $W$ which do not appear in the image of $\psi$.

 Therefore the map $\varphi$ is a $\mathcal G$-isomorphism from $V/\im \psi$ into the non-trivial part of $U$. Moreover, the map $\varphi\circ \lambda \colon U \to U$ has the same scalar factors $c_W(\varphi\circ \lambda) = c_W(\lambda\circ\varphi)\neq 0$ and on the trivial part it is the scalar multiplication by $(p^2-1)/4\neq 0$. It follows that $\varphi\circ\lambda$ is a $\mathcal G$-automorphism of $U$.


\subsection{End of proof of Proposition~\ref{pairing_prop}}\label{pairing_proof_subsec}

 We compute
 \begin{equation*}
   \varphi\circ\lambda\,(\mathfrak v) =
   \varphi\Bigl( \sum_{\substack{ \{A,B\}\text{ with}\\ A\antip B\in \mathfrak v}} \{A,B\}\Bigr)
   = \sum_{\substack{\{A,B\}\text{ with}\\ A\antip B\in \mathfrak v}}\ \sum_{\substack{\mathfrak w\text{ with}\\ A\antip B\in \mathfrak w}} \mathfrak w
   = \sum_{\mathfrak w} \langle \mathfrak v,\mathfrak w\rangle \cdot \mathfrak w.
 \end{equation*}
 We deduce from the above representation theoretic input that the pairing in Section~\ref{pairing_subsec} is non-degenerate.

 This concludes the proof of Proposition~\ref{pairing_prop}. It is to note that the non-degeneracy of the pairing is equivalent to the difficult part of the proof of Chen's isogeny in Theorem~\ref{chen_thm}. It would be nice to find a purely combinatorial proof of the non-degeneracy of this pairing.


\subsection{End of proof of Theorem~\ref{chen_thm}}

 In Section~\ref{grth_subsec}, we have shown that the map $\varphi\circ \lambda \colon \bigoplus_{\mathfrak v} \ZZ\,(E,\mathfrak v)\to\bigoplus_{\mathfrak v} \ZZ\,(E,\mathfrak v)$ has finite kernel and cokernel. In other words the map
 \begin{equation*}
   \varphi\circ\lambda \colon \divi(\Xn^+) \to  \divi(\Xn^+)
 \end{equation*}
 has finite kernel and cokernel in each fibre. Since the size of them is independent of the fibre, the above map has kernel and cokernel of finite exponent. Now consider the induced map
 \begin{equation*}
   \varphi\circ\lambda \colon \Jac(\Xn^+) \to  \Jac(\Xn^+)
 \end{equation*}
 on the Jacobian. If $[D]$ is a divisor class in $\Jac(\Xn^+)$, then there is a multiple $[mD]$ which is in the image of $\varphi\circ\lambda$. Therefore the map $\varphi\circ\lambda$ has finite cokernel on the Jacobians. Comparing the dimensions it follows that it has finite kernel, too. This implies that $\varphi$ has finite cokernel and $\lambda$ has finite kernel in the sequences in Theorem~\ref{chen_thm}.

 To conclude we have to verify that the sequences have finite cohomology in the middle term. This can be deduced from counting the dimension together with all the known parts of the theorem: We know from~\eqref{genus_eq} in Section~\ref{genus_subsec} that the dimension of $\Jac(\Xs^+)$ is equal to the sum of the dimensions of $\Jac(X_0)$ and $\Jac(\Xn^{+})$. Since $\varphi$ has finite cokernel, its kernel has now the same dimension as $\Jac(X_0)$, which is also the dimension of the image of $\psi$. Because the sequence is a complex, we have $\im \psi\subset \ker\varphi$ and the quotient is finite because they have the same dimension. The argument for the second sequence is similar. This concludes the proof of Theorem~\ref{chen_thm}.


\subsection{Relation to Chen's computations}\label{chen_grth_subsec}

 In this subsection, we relate our proof to the previous proof in~\cite{chen2} by  establishing a translation. The first difference is that we work with $\PGL_2(\Fp)$ rather than with $\GL_2(\Fp)$, but that does not make any real difference.

 Fix an elliptic curve $E$ over an algebraically closed field of characteristic different from $p$. Let us fix two distinct subgroups $A_0$ and $B_0$ in $E$. Further we choose a necklace $\mathfrak v_0$ in which $A_0$ is antipodal to $B_0$. Let $\mathcal{B}$ be the stabiliser of $A_0$ in $\mathcal G= \PGL\bigl(E[p]\bigr)$, which is a Borel subgroup, let $\mathcal S$ be the stabiliser of $\{A_0,B_0\}$, which is the normaliser of a split Cartan subgroup, and let $\mathcal N$ be the stabiliser of $\mathfrak v_0$, which is the normaliser of a non-split Cartan. Then we define  three $\mathcal G$-isomorphisms
 \begin{equation*}
   \iota_0 : \QQ\bigl[\mathcal{G}/\mathcal{B}\bigr]  \longrightarrow \bigoplus_A \QQ\,A\, ,
       \qquad
     \ios : \QQ\bigl[\mathcal G/\mathcal S\bigr]   \longrightarrow   \bigoplus_{\{A,B\}} \QQ \, \{A,B\} \, ,\qquad
     \ion :  \QQ\bigl[\mathcal G/\mathcal N\bigr]
\longrightarrow   \bigoplus_{\mathfrak v} \QQ\,\mathfrak v
 \end{equation*}

 by $\iota_0(\mathcal B)=A_0$, $\ios(\mathcal S ) = \{A_0,B_0\}$ and $\ion(\mathcal N)= \mathfrak v_0$. The importance of the exact choices here is that $\mathcal N\cap \mathcal S$ contains $4$ elements. Had we taken ``adjacent'' rather than ``antipodal'' pearls in the necklace, we would only have $2$ elements. Compare with remarque~3 in~\cite{desmit_edixhoven}.

 Recall from~\cite{chen2} that for each double coset $\mathcal Hg\mathcal H'$ for some subgroups $\mathcal H$ and $\mathcal H'$ of $\mathcal G$ and $g\in \mathcal G$, there is a $\mathcal G$-morphism $\Theta(\mathcal Hg\mathcal H')\colon \QQ[\mathcal G/\mathcal H]\to\QQ[\mathcal G/\mathcal H']$ sending $\mathcal H$ to the sum $\sum_{s\in\Omega} s\mathcal H'$ such that $\bigcup_{s\in\Omega} s\mathcal H' = \mathcal Hg \mathcal H'$ is a disjoint union.
 \begin{lem}
   We have
   \begin{equation*}
     \psi = \ios\circ\Theta(\mathcal{B} 1 \mathcal{S})\circ\iota_0^{-1}, \quad \mu= \iota_0\circ\Theta(\mathcal{S}1\mathcal{B})\circ\ios^{-1},\quad
     \varphi = \ion\circ\Theta(\mathcal{S}1\mathcal{N})\circ\ios^{-1}, \quad\text{and}\quad \lambda = \ios\circ\Theta(\mathcal{N}1\mathcal{S})\circ \ion^{-1}.
   \end{equation*}
   Further we have $\alpha = \ios\circ\Theta(\mathcal{S} g \mathcal{S})\circ\ios^{-1}$ with $g = [\ma{1}{1}{1}{-1}]$.
 \end{lem}

 \begin{proof}
   We only illustrate the first equality as the proof is very similar for all of the first four equalities. The map $\Theta(\mathcal{B} 1 \mathcal{S})$ sends $\mathcal{B}=\iota_0^{-1}(A_0)$ to the sum of $s\mathcal S$ where $s$ runs over a system $\Omega$ of representatives of $\mathcal{B}/(\mathcal{B}\cap \mathcal{S})$.  The quotient group is the group of elements in $\mathcal G$ fixing $A_0$ modulo the subgroup of elements also fixing $B_0$. So
   \begin{equation*}
     \ios\circ\Theta(\mathcal{B} 1 \mathcal{S})(\mathcal B)=\sum_{s\in\Omega} \{A_0,sB_0\}.
   \end{equation*}
   Since $\mathcal{B}$ acts transitively on  $\PP\bigl(E[p]\bigr)\setminus\{A_0\}$, each $\{A_0,B\}$ with $B\neq A_0$ will appear exactly once in this sum. Hence
   \begin{equation*}
     \ios\circ \Theta(\mathcal{B} 1 \mathcal{S})(\mathcal{B}) = \sum_{B \neq A_0} \{A_0,B\} = \psi(A_0).
   \end{equation*}
   To prove the last equality, note that with $C_0=gA_0$ and $D_0=gB_0,$ we get $[A_0,B_0;C_0,D_0]=-1$ and
   $g\mathcal{S}g^{-1}$ is the stabiliser in $\mathcal G$ of $\{C_0,D_0\}$.
   So the quotient $\mathcal S/(\mathcal S\cap g\mathcal S g^{-1})$ is the group of elements fixing $\{A_0,B_0\}$ modulo elements also fixing $\{C_0,D_0\}$. It follows that
   \begin{equation*}
     \ios\circ\Theta(\mathcal{S}g\mathcal{S})(\mathcal{S})=\sum_{s\in\Omega'}sg\{A_0,B_0\}=\sum_{s\in\Omega'}\{sC_0,sD_0\}
   \end{equation*}
   where $\Omega'$  is a system of representatives of $\mathcal S/(\mathcal S\cap g\mathcal S g^{-1})$. This is exactly the sum of all $\{C,D\}$ with $[A_0,B_0;C,D]=-1$, because the action of $\mathcal S$ on the set of pairs $\{C,D\}$ is transitive and for $s\in\Omega$, we have $[A_0,B_0;sC_0,sD_0]=[sA_0,sB_0;sC_0,sD_0]=[A_0,B_0;C_0,D_0]=-1$.
 \end{proof}

 Now it is clear that equation~\eqref{chen86_eq} is exactly what Chen proves in Proposition~8.6 and Proposition~8.7 in\cite{chen2}. His proof is a computation in double coset operators. He then goes on to give formulae for the values of $c_W(\lambda\circ\varphi)$ in terms of character sums. However, his final argument that they are non-zero can be shortened as we did in Section~\ref{grth_subsec} without making the values more explicit.

 Finally, we wish to point out that Chen also describes the maps using the degeneracy morphisms. See his Theorem~2 in~\cite{chen2}. For instance, let consider the usual degeneracy morphisms $\pi_0:X_\mathcal A\longrightarrow X_0$ defined by  $(E,(A,B,C))\mapsto (E,A)$ and $\pis^+ :X_\mathcal A\longrightarrow \Xs^+$ defined by  $(E,(A,B,C)\mapsto (E,\{A,B\})$,  it is easy to see from our definitions that
 \begin{equation*}
   (p-1)\cdot \psi = (\pis^{+})_* \circ(\pi_0)^*\qquad\text{ and }\qquad
   (p-1)\cdot \mu = (\pi_0)_*\circ(\pis^{+})^*
 \end{equation*}
 hold as maps on divisors. To explain $\varphi$ and $\lambda$, we have to replace $\pin^+$ by another degeneracy map. Let $\varepsilon$ be a non-square in $\Fp$. We define $\piin\colon X_{\mathcal A}\to \Xn^{+}$ by sending $\bigl( E, (A,B,C)\bigr)$ to the following necklace $\mathfrak v$ in $E$. First there exist a unique $D$ distinct from $A$, $B$, and $C$ such that $[A,B;C,D] = \varepsilon$. Then, by Lemma~\ref{antip_pairing_lem} there is a unique $\mathfrak v$ such that $A\antip B \in \mathfrak v$ and $C\antip D \in \mathfrak v$. It is also this lemma which shows that this map is $\PGL\bigl(E[p]\bigr)$-equivariant.
 \begin{lem}
   We have
   \begin{equation*}
     4\cdot \varphi  = (\piin)_*\circ (\pis^{+})^* \qquad\text{ and }\qquad
     4\cdot \lambda = (\pis^{+})_*\circ(\piin)^*.
   \end{equation*}
 \end{lem}
 \begin{proof}
   Let $A$ and $B$ be two distinct cyclic subgroups of order $p$ of some elliptic curve $E$. By definition, we have
   \begin{align*}
     (\piin)_*\circ (\pis^{+})^* \bigl( E, \{A,B\}\bigr) &=
      \sum_{C\not\in\{A,B\}} \piin\bigl( E, (A,B,C) \bigr) +  \sum_{D\not\in\{A,B\}}  \piin\bigl( E, (B,A,D)\bigr)\\
      &=2\sum_{X\not\in\{A,B\}} \piin\bigl( E, (A,B,X) \bigr)
   \end{align*}
   since $[A,B;C,D]=[B,A;D,C]$ for all $C,D$. Each necklace in this sum will have $A$ and $B$ as antipodal pearls. Let $\mathfrak v$ be a necklace with $A\antip B\in \mathfrak v$. We wish to determine how often $\mathfrak v$ appears in the above sum, that is to say how many $X\not\in\{A,B\}$ are there such that $\mathfrak v=\piin\bigl( E, (A,B,X) \bigr) $. In other words, we wish to count the $X$ such that $[A,B;X,X']= \varepsilon$ where $X'$ is the antipodal pearl to $X$ in $\mathfrak v$.  We can choose a basis of $E[p]$ such that $A=\infty$, $B=0$ and the subgroups $X\not\in\{A,B\}$ are $X=1/a$ for some $a\in\Fp^\times$. The involution in the stabiliser of $\mathfrak v$ is then represented by a matrix $g=[\ma{0}{d}{1}{0}]$ with $d$ non-square and the antipodal pearl to $X$ in $\mathfrak v$ is $X'=da$. It follows that $[A,B;X,X']=da^2$. Since $\varepsilon$ and $d$ are non-squares, the equation $da^2=\varepsilon$ has two solutions in $\Fp^\times$. Hence there are two pearls $X\not\in\{A,B\}$ such that $\mathfrak v=\piin\bigl( E, (A,B,X) \bigr) $ and consequently
   \begin{equation*}
     (\piin)_*\circ (\pis^{+})^* \bigl( E, \{A,B\}\bigr) = 4\sum_{\substack{\mathfrak v \text{ with}\\A\antip B\in\mathfrak v}} \mathfrak v.
   \end{equation*}
   The second equality follows from an analogous argument.
 \end{proof}


\section{Examples}\label{examples_sec}

 We add some numerical examples for small primes, mainly on the eigenvalues of the pairing in Section~\ref{pairing_subsec}.


\subsection{Necklaces for $p=5$}

 There is a unique conjugacy class $C_{\gamma}$ in $\PGL_2(\FF_5)$. We have $t=1$ and $n=2$. We spell out the 10 necklaces below by giving them as a list of all points in $\PP^1(\FF_5)$:
 \begin{gather*}
   \bigl( 0,1,2,4,\infty,3 \bigr),\
   \bigl( 0,1,3,\infty,2,4 \bigr),\
   \bigl( 0,1,4,2,3,\infty \bigr),\
   \bigl( 0,1,\infty,3,4,2 \bigr),\
   \bigl( 0,2,1,\infty,4,3 \bigr),\\
   \bigl( 0,3,1,2,\infty,4 \bigr),\
   \bigl( 0,3,2,1,4,\infty \bigr),\
   \bigl( 0,2,3,4,1,\infty \bigr),\
   \bigl( 0,2,\infty,1,3,4 \bigr),\
   \bigl( 0,4,1,3,2,\infty \bigr).
 \end{gather*}
 It is now easy to read off the pairing $\langle\cdot,\cdot\rangle$ defined in Section~\ref{pairing_subsec}. Let $\mathfrak v$ be the first necklace in the list. Of course, we have $\langle \mathfrak v,\mathfrak v\rangle = 3$. On the one hand, we have $\langle\mathfrak v,\mathfrak w\rangle = 0$ for $\mathfrak w$ being any of the necklaces from the second to the seventh and, on the other hand, $\langle\mathfrak v,\mathfrak w\rangle = 1$ when $\mathfrak w$ is any of the last three necklaces.

 The resulting matrix $(\langle \mathfrak v, \mathfrak w\rangle)_{\mathfrak v, \mathfrak w}$ is non-singular. Its eigenvalues are $6$, four times $1$, and five times $4$.


\subsection{Necklaces for $p=7$}

 For $p=7$, we have two choices for $\gamma$. We take $t=1$ and $n=3$, here.
 \begin{gather*}
   \bigl(0, \infty , 2, 3, 5, 1, 4, 6\bigr),\
   \bigl(0, 6, 3, 5, \infty , 2, 4, 1\bigr),\
   \bigl(0, \infty , 3, 1, 4, 5, 6, 2\bigr),\
   \bigl(0, \infty , 1, 5, 6, 4, 2, 3\bigr),\\
   \bigl(0, 3, \infty , 2, 5, 4, 6, 1\bigr),\
   \bigl(0, 3, 4, 5, 1, 6, \infty , 2\bigr),\
   \bigl(0, 2, 1, 4, \infty , 3, 6, 5\bigr),\
   \bigl(0, 2, 3, \infty , 5, 6, 1, 4\bigr),\\
   \bigl(0, 5, 4, \infty , 2, 1, 6, 3\bigr),\
   \bigl(0, 5, \infty , 1, 6, 2, 3, 4\bigr),\
   \bigl(0, 3, 5, 6, \infty , 1, 2, 4\bigr),\
   \bigl(0, 5, 1, 2, 3, 6, 4, \infty \bigr),\\
   \bigl(0, 3, 1, \infty , 4, 2, 5, 6\bigr),\
   \bigl(0, 1, \infty , 3, 4, 6, 2, 5\bigr),\
   \bigl(0, \infty , 5, 4, 2, 6, 3, 1\bigr),\
   \bigl(0, 2, 4, 3, 6, \infty , 5, 1\bigr),\\
   \bigl(0, 6, 2, \infty , 1, 4, 3, 5\bigr),\
   \bigl(0, 4, \infty , 5, 2, 3, 1, 6\bigr),\
   \bigl(0, \infty , 6, 2, 1, 3, 5, 4\bigr),\
   \bigl(0, 4, 6, \infty , 3, 5, 2, 1\bigr),\\
   \bigl(0, 6, \infty , 4, 3, 1, 5, 2\bigr).
 \end{gather*}
 Again the pairing is degenerate with eigenvalues $12$, six times $4+2\sqrt{2}$, six times $4-2\sqrt{2}$, and eight times $3$.


\subsection{Larger primes}

 We list the characteristic polynomial of the matrix $(\langle \mathfrak v,\mathfrak w\rangle)_{\mathfrak v,\mathfrak w}$ for the next few primes.
 \begin{center}
   \begin{tabular}{cl}
     $p$ & char. polynomial of $\langle\cdot,\cdot\rangle$ \\ \hline \\[-1ex]
     11 &  $(X - 30) \cdot (X - 2)^{10} \cdot (X - 8)^{20} \cdot (X^2 - 10\, X + 5)^{12}$ \\[1mm]
     13 & $(X - 42) \cdot (X^3 - 19\,X^2 + 83\,X - 1)^{12} \cdot (X - 12)^{14} \cdot (X - 4)^{27} $ \\[1mm]
     17 & $(X - 72) \cdot (X - 1)^{16} \cdot (X^3 - 27\, X^2 + 195\, X - 361)^{16} \cdot (X - 16)^{17} \cdot $ \\
        & \ $\cdot (X - 8)^{18} \cdot (X^2 - 16\,X + 32)^{18}  $ \\[1mm]
     19 & $(X - 90) \cdot (X - 18)^{18} \cdot
     (X^4 - 32\, X^3 + 304\, X^2 - 768\, X + 256)^{18} \cdot $\\
         & \ $ \cdot (X - 3)^{20} \cdot (X^3 - 33\, X^2 + 315\, X - 867)^{20} $
   \end{tabular}
 \end{center}
 These values for the eigenvalues $c_W(\varphi\circ\lambda)$ coincide with Chen's computation in his table~2 in~\cite{chen2}.


\bibliographystyle{amsalpha}
\bibliography{xnonsplit}

\end{document}